\documentclass[11pt,a4paper,reqno]{amsart}

\usepackage [english]{babel}  

\DeclareFontFamily{U}{mathb}{\hyphenchar\font45}
\DeclareFontShape{U}{mathb}{m}{n}{
      <5> <6> <7> <8> <9> <10>
      <10.95> <12> <14.4> <17.28> <20.74> <24.88>
      mathb10
      }{}
\DeclareSymbolFont{mathb}{U}{mathb}{m}{n}

\DeclareMathSymbol{\sqbullet}{1}{mathb}{"0D}

\usepackage{amsmath,amssymb}
\usepackage{mathtools} 

\usepackage{euscript}
\usepackage{graphicx}
\usepackage[all]{xy}
\usepackage[dvipsnames]{xcolor}
\usepackage[colorlinks=true,linktocpage=true,pagebackref=false, citecolor=black,linkcolor=black]{hyperref} 
\usepackage{pifont}
\usepackage{tikz,pgfplots}
\usepackage{pgffor}

\pgfplotsset{compat=1.17}
\usepgfplotslibrary{fillbetween}
\usetikzlibrary{cd,arrows,decorations.pathmorphing,backgrounds,automata,positioning,fit,matrix}
\usepackage{caption,subcaption}
\usepackage{comment}

\allowdisplaybreaks

\pgfplotsset{soldot/.style={color=black,only marks,mark=*}} \pgfplotsset{holdot/.style={color=black,fill=white,only marks,mark=*}}

\newtheorem{thm}{Theorem}[section]

\newtheorem{lem}[thm]{Lemma}
\newtheorem{prop}[thm]{Proposition}

\newtheorem{cor}[thm]{Corollary}

\newtheorem{prob}[thm]{Problem}

\theoremstyle{definition}
\newtheorem{defn}[thm]{Definition}

\theoremstyle{remark}
\newtheorem{remark}[thm]{Remark}

\newtheorem{example}[thm]{Example}
\newtheorem{examples}[thm]{Examples}

\numberwithin{equation}{section}
\numberwithin{figure}{section}

\newcommand{\Z}{{\mathbb Z}} \newcommand{\R}{{\mathbb R}}

\newcommand{\sph}{{\mathbb S}}

 \newcommand{\Cont}{{\mathcal C}}

\newcommand{\Ff}{{\EuScript F}}

\newcommand{\Tt}{{\EuScript T}}
\newcommand{\Ee}{{\EuScript E}}

\newcommand{\pol}{{\EuScript K}}
\newcommand{\Qq}{{\EuScript Q}}

\newcommand{\Ll}{{\EuScript L}}

\newcommand{\Int}{\operatorname{Int}}
\newcommand{\cl}{\operatorname{Cl}}
\newcommand{\dist}{\operatorname{dist}}
\newcommand{\id}{\operatorname{id}}

\newcommand{\sd}{\operatorname{sd}}
\newcommand{\diam}{\operatorname{diam}}
\newcommand{\Conv}{\operatorname{Conv}}
\newcommand{\st}{\operatorname{St}}

\newcommand{\x}{{\tt x}} \newcommand{\y}{{\tt y}} 
\newcommand{\z}{{\tt z}} 
\newcommand{\s}{{\tt s}}

\newcommand{\veps}{\varepsilon}
\newcommand{\eps}{\epsilon}
\newcommand{\ol}{\overline}

\usepackage{scalerel}

\newcommand\ov[1]{\overset{\circ}{#1}}

\begin{document}

\title[Differentiable approximation of continuous definable maps]{Differentiable approximation of continuous definable maps that preserves the image}

\author{Antonio Carbone}
\address{Dipartimento di Matematica, Via Sommarive, 14, Universit\`a di Trento, 38123 Povo (ITALY)}
\email{antonio.carbone@unitn.it}

\date{4/6/2024}

\begin{abstract}
Recently Paw\l{}ucki showed that compact sets that are definable in some o-minimal structure admit triangulations of class $\Cont^p$ for each integer $p\geq 1$. In this work, we make use of these new techniques of triangulation to show that all continuous definable maps between compact definable sets can be approximated by differentiable maps without changing their image. The argument is an interplay between o-minimal geometry and PL geometry and makes use of a `surjective definable version' of the finite simplicial approximation theorem that we prove here.
\end{abstract}

\maketitle

\section{Introduction}

Approximation of continuous maps is an important tool in many areas of mathematics. In particular, in geometry the possibility of approximating continuous maps by a dense subclass of maps with a better behaviour often allows a deeper and better understanding of many properties. A famous example is the classical Weierstrass approximation theorem: \textit{Every continuous map $f:X\to\R^m$, defined on a compact subset $X$ of $\R^n$, can be uniformly approximated on $X$ by polynomial maps.} Here, the approximating maps take values in the Euclidean space $\R^m$, but several difficulties arise when one tries to restrict the image of the approximating map to a fixed target space $Y\subset\R^m$. For instance, there exist no nonconstant polynomial maps from the sphere $\sph^2$ to the circle $\sph^1$ \cite[Thm.13.1.9]{bcr}. 

A deep problem in real algebraic geometry regards the possibility of approximating continuous (or smooth) maps between nonsingular real algebraic sets by regular maps. Besides the Weierstrass approximation theorem quoted above, where the target space $Y$ is an Euclidean space, determining when a continuous map between real algebraic sets can be approximated by a regular map is a hard problem. For instance, except for the cases $m=1,2,4$ studied mainly by Bochnak and Kucharz \cite{ 1,2, 3}, it is not known whether a continuous map between a sphere $\sph^n$ of dimension $n$ and a sphere $\sph^m$ of dimension $m$ could be approximated by a regular map or not. We refer the reader to \cite[\S 12, \S13.3]{bcr} and the survey \cite{bk4} for further informations about this fascinating problem.  It is worthwhile to remark that compact real algebraic sets with at least two points do not have tubular neighbourhoods with regular retractions \cite[Thm.2]{gh}. If the target space $Y$ is `generic' in a suitable sense, then regular maps from a compact nonsingular real algebraic set $X$ to $Y$ are `very few', see \cite{gh2,gh3}. This lack of regular maps seems to be one of the main difficulties for an extension of the Nash-Tognoli algebraization techniques \cite{t} from smooth manifolds to singular polyhedra, see \cite{AKi, ght}. 

To overcome these difficulties one can consider, instead of polynomials or regular maps, a more flexible class of approximating maps like differentiable maps. Let $X\subset\R^n$ and $Y\subset\R^m$ be sets and $p\geq 1$ a positive integer. A map $f:X\to Y$ is \textit{of class $\Cont^p$} if it admits an extension $\widehat{f}:\Omega\to\R^m$, defined on an open neighbourhood $\Omega$ of $X$ in $\R^n$, which is of class $\Cont^p$ in the usual sense. A classical example of differentiable approximation concerns Whitney's approximation theorem \cite{w} for continuous maps whose target space is a $\Cont^p$ submanifold of $\R^m$. Here, a crucial point is the existence of $\Cont^p$ tubular neighbourhoods for $\Cont^p$ submanifolds in $\R^m$. When the target space $Y$ has `singularities', the lack of existence of tubular neighbourhoods (of class $\Cont^p$ for $p\geq 1$), makes the approximation problem difficult to approach in its full generality. One possible way to proceed is to consider domains of definitions and target spaces in some suitable tame category. For instance, Fernando and Ghiloni \cite{fgh2} conducted an extended study on differentiable approximation for continuous maps when the target space $Y$ admits suitable triangulations. They showed that differentiable approximation is possible for a wide class of triangulable sets including differentiable manifolds, polyhedra, semialgebraic sets, subanalytic sets and more generally definable sets of an o-minimal structure. 

In this work we focus on differentiable approximation for continuous maps that are definable in an o-minimal structure. An \textit{o-minmal structure} (on the ordered field of real numbers $\R$) is a collection $\mathfrak{S}:=\{\mathfrak{S}\}_{n\in\mathbb{N}^*}$ of families of subsets of $\R^n$ satisfying the following properties:
\begin{itemize}
\item $\mathfrak{S}_n$ is a Boolean algebra,
\item $\mathfrak{S}_n$ contains the algebraic subsets of $\R^n$,
\item if $X\in \mathfrak{S}_n$ and $Y\in\mathfrak{S}_m$, then $X\times Y\in\mathfrak{S}_{n+m}$,
\item if $\pi:\R^n\times\R\to\R^n$ is the projection onto the first factor and $X\in\mathfrak{S}_{n+1}$, then $\pi(X)\in\mathfrak{S}_n$,
\item $\mathfrak{S}_1$ consists exactly of all the finite unions of points and intervals (of any type). 
\end{itemize}
The elements of $\mathfrak{S}_n$ are called \textit{definable subsets of} $\R^n$. A map $f:X\to Y$, between a definable subset $X\subset\R^n$ and a definable subset $Y\subset \R^m$, is a \textit{definable map} if its graph $\Gamma_f$ is a definable subset of $\R^{n+m}$. As a consequence of the Tarski-Seidenberg theorem \cite[Thm.2.2.1]{bcr}, semialgebraic sets constitute an o-minimal structure, which is the `smallest' o-minimal structure, in the sense that it is contained in any other o-minimal structure. In particular, semialgebraic sets and maps are definable in any o-minimal structure. The collection of global subanalytic sets is precisely the collection of definable sets in the o-minimal structure $\R_{\text{an}}$ (see \cite{wi}). We refer the reader to \cite{Co, vD, vDM} for further information on the theory of o-minimal structures. For the rest of this article, even if not explicitly mentioned, when we refer to definable sets or definable maps we mean definable in a fixed o-minimal structure. 

Let $X\subset\R^n$ be a compact set and $Y\subset\R^m$ any set. Given a continuous map $f:X\to Y$ we denote by $\|f\|$ the uniform norm of $f$, that is 
$$
\|f\|:=\max_{x\in X}\{|f(x)|_m : x\in X\},
$$
where $|\cdot|_m$ denotes the Euclidean norm of $\R^m$. We denote by $\Cont^0(X,Y)$ the set of all continuous maps from $X$ to $Y$ endowed with the compact-open topology. A fundamental system of open neighbourhoods of $f\in\Cont^0(X,Y)$ is given by the sets 
$$
\mathcal{N}(f,\veps):=\{g\in\Cont^0(X,Y) : \|f-g\|<\veps\},
$$
where $\veps>0$ is any strictly positive number. Let $p\geq 1$ be an integer and assume that the sets $X\subset\R^n$ and $Y\subset\R^m$ are definable. A definable map $f:X\to Y$ is \textit{of class $\Cont^p$} if it admits a definable extension $\widehat{f}:\Omega\to\R^m$ of class $\Cont^p$ defined on an open definable neighbourhood $\Omega$ of $X$ in $\R^n$. 

In \cite{fgh2} the authors investigated differential approximation for continuous maps in a very general context and their results do not take into account whether the involved continuous map $f$ is definable or not. Therefore, it is not ensured that the approximating map is definable in case that the map $f$ is definable. When one deals with maps that are definable in an o-minimal structure it is natural to wonder if it is possible to preserve this additional structure after the approximation. The general problem is the following:

\begin{prob}[Definable approximation problem]\label{approxprob2}
Let $p\geq 1$ be an integer. Determine which definable subsets $Y\subset\R^m$ have the following property: For each $\veps>0$, each compact definable subset $X$ of each Euclidean space $\R^n$ and each continuous definable map $f:X\to Y$, there exists a definable map $g:X\to Y$ of class $\Cont^p$ such that $\|f-g\|<\veps$.
\end{prob}

In \cite{fgh1} the authors focused on the semialgebraic case. Even if not explicitly mentioned in the paper, the results of \cite{fgh1} hold true (with the obvious modifications to the proofs) in every o-minimal structure. In particular, as a consequence of the existence of $\Cont^1$-triangulations for definable sets \cite{os}, they showed that Problem \ref{approxprob2} has a positive answer for each definable set $Y$ in the case $p=1$, see \cite[Thm.1.3]{fgh1}. Recently, Paw\l{}ucki \cite{p1,p2} improved the results of \cite{os}, showing that definable sets admit triangulations of class $\Cont^p$ for each $p\geq 1$, see \S\ref{pawu}. As an application of his remarkable results and \cite[Cor.1.5]{fgh1}, he provided a complete answer to Problem \ref{approxprob2} for every integer $p\geq 1$, see also the proof of \cite[Thm.1.4]{fgh1}.

\begin{thm}[{\cite[Thm.9.2]{p1}}]\label{approxPaw1}
Let $X$ be a compact definable subset of $\R^n$, $Y$ a definable subset of $\R^m$ and $f:X\to Y$ a continuous definable map. Let $\veps>0$ and let $p\geq 1$ be an integer. Then, there exists a definable map $g:X\to Y$ of class $\Cont^p$ such that $\|f-g\|<\veps$ and $g(X)\subset f(X)$. 
\end{thm}

For a general target space $Y$, the previous theorem does not hold true for $p=\infty$. Already in the semialgebraic case, the problem of determining for which semialgebraic sets $Y\subset\R^m$ it is possible to approximate a continuous semialgebraic map with values in $Y$ by a Nash map (i.e. a $\Cont^{\infty}$ semialgebraic map) is a very difficult task. Shiota \cite{sh} showed that continuous semialgebraic maps can be approximated by Nash maps when the target space $Y$ is an affine Nash manifold (i.e. a semialgebraic subset of some $\R^m$ that is also a smooth submanifold). As a consequence of Theorem \ref{approxPaw1} and the results of Baro, Fernando and Ruiz \cite[Thm.1.7]{bfr}, it holds that a continuous semialgebraic map between compact Nash sets with only normal crossings (and more generally monomial) singularities and which preserves the irreducible components can be approximated by a Nash map. We refer the reader to \cite[\S1]{bfr} for the precise definitions and further information. Recently, we showed \cite[Thm.1.12]{cf} that Nash approximation for continuous semialgebraic maps is possible when the target space is an affine Nash manifold with (divisorial) corners. An \textit{affine Nash manifold with corners} is a semialgebraic subset of some $\R^m$ that is also a submanifold with corners (see \cite{jo} for further information about manifolds with corners). Besides these cases, very little is known about Nash approximation of continuous semialgebraic maps. The main difficulty lies in the fact that Nash maps are analytic and of algebraic nature \cite[Prop.8.1.8]{bcr}, so one cannot use the standard tools in approximation theory, typical of the smooth category, like partitions of unity or integration of vector fields. The following example shows that, in general, a continuous semialgebraic map cannot be approximated by Nash maps:

\begin{example}
Let $X:=[0,1]$ and $Y:=\{\x\y=0\}\subset\R^2$. Consider the continuous semialgebraic map $f:X\to Y$ defined as follows:
$$
f(x):=
\begin{cases}
\big(\tfrac{1}{2}-x,0\big), \text{ if } 0\leq x\leq \tfrac{1}{2},\\
\big(0,x-\tfrac{1}{2}\big), \text{ if } \tfrac{1}{2}< x\leq 1.
\end{cases}
$$
Let $g:X\to Y$ be any Nash map. As Nash maps are analytic, then either $g(X)\subset\{\x=0\}$ or $g(X)\subset \{\y=0\}$. We deduce, $\|f-g\|\geq \tfrac{1}{2}$. Thus, $f$ cannot be approximated by a Nash map. $\sqbullet$
\end{example}

Theorem \ref{approxPaw1}, as well as the mentioned results of Fernando and Ghiloni, only takes into account the target space $Y$ during the approximation process and do not guarantee that $g(X)=f(X)$. The purpose of this paper is to improve Theorem \ref{approxPaw1} in order to have an approximating map $g$ that not only has $Y$ as target space, but that has also the same image of $f$. 

Our result reads as follows.

\begin{thm}\label{main}
Let $f:X\to \R^m$ be a continuous definable map defined on a compact definable set $X\subset\R^n$. Let $\veps>0$ and let $p\geq 1$ be an integer. Then, there exists a definable map $g:X\to \R^m$ of class $\Cont^p$ such that $\|f-g\|<\veps$ and $g(X)=f(X)$. 
\end{thm}

The proof of this theorem is an interplay between o-minimal geometry and PL geometry and makes use of the following `surjective definable version' of the finite simplicial approximation theorem, which seems interesting in its own right. In general, the classical finite approximation theorem \cite[Thm.16.1]{mu} does not guarantee that a simplicial approximation $h$ is surjective if the involved continuous map $f$ is surjective, see Examples \ref{notsurj} below. If $\pol$ is a finite simplicial complex of $\R^n$, we indicate with $\sd^\kappa(\pol)$ its $\kappa^{\text{th}}$ barycentric subdivision, see \S\ref{finitsim}.

\begin{prop}[Surjective simplicial approximation]\label{key3}
Let $\pol$ be a finite simplicial complex of $\R^n$, $\Ll$ a finite simplicial complex of $\R^m$ and $f:|\pol|\to|\Ll|$ a surjective continuous definable map. Then, for each $\veps>0$ there exist two integers $\kappa,\ell\geq 0$ and a surjective simplicial map $h:|\sd^{\kappa}(\pol)|=|\pol|\to|\sd^{\ell}(\Ll)|$ such that $\|f-h\|<\veps$.
\end{prop}

\section{Preliminaries}

In this section we present some preliminary definitions and results that will be freely used throughout this article.

\subsection{Finite simplicial complexes}\label{finitsim}

A \textit{simplex} $\sigma\subset\R^n$ of dimension $d$ is the convex hull of $d+1$ affinely independent points $\nu_0,\ldots,\nu_d\in\R^n$, that is
$$
\sigma=\Conv(\{\nu_0,\ldots,\nu_d\}):=\Big\{\lambda_0\nu_0+\ldots+\lambda_d\nu_d :  \lambda_0\geq 0,\ldots,\lambda_d\geq 0, \sum_{i=0}^d\lambda_i=1\Big\}.
$$
If $0\leq i_0<\ldots<i_k\leq d$, the simplex $\Conv(\{\nu_{i_0},\ldots,\nu_{i_k}\})$ is called a \textit{face of} $\sigma$ of dimension $k$. As usual, a face of dimension $d-1$ is called a \textit{facet of} $\sigma$ while a face of dimension zero is called a \textit{vertex of} $\sigma$. We denote $\partial\sigma$ the \textit{(relative) boundary of} $\sigma$ defined as the union of the proper faces of $\sigma$ and $\ov{\sigma}:=\sigma\setminus\partial\sigma$ the \textit{(relative) interior of} $\sigma$, which is equal to the interior of $\sigma$ in the affine space generated by $\sigma$ in $\R^n$ (that is, the smallest affine subspace of $\R^n$ that contains $\sigma$). Observe that if $\sigma$ is a simplex of dimension zero, then $\partial \sigma=\varnothing$, so $\ov{\sigma}=\sigma$.

A \textit{finite simplicial complex $\pol$ of $\R^n$} is a finite family of simplices of $\R^n$ such that 
\begin{itemize}
\item for each simplex $\sigma\in\pol$ all the faces of $\sigma$ belong to $\pol$,
\item for each $\sigma_1,\sigma_2\in\pol$ the intersection $\sigma_1\cap\sigma_2$ is either empty or a common face of both $\sigma_1$ and $\sigma_2$.
\end{itemize}
The \textit{(underlying) polyhedron $|\pol|$ of} $\pol$ is the set $\bigcup_{\sigma\in\pol}\sigma$ equipped with the topology induced by the Euclidean topology of $\R^n$. Observe that $|\pol|$ is always a compact subset of $\R^n$. A simplex $\sigma\in\pol$ is a \textit{maximal simplex of} $\pol$ if it is not a proper face of another simplex $\sigma'\in\pol$. Equivalently, $\sigma\in\pol$ is a maximal simplex of $\pol$ if $\sigma$ satisfies the following property: for each $\sigma'\in\pol$, it holds $\ov{\sigma}\cap\sigma'\neq \varnothing$ if and only if $\sigma=\sigma'$. We indicate the set of vertices of $\pol$ with $\pol_{\bullet}$. Note that if $\sigma\in\pol_{\bullet}$ is a vertex that is also a maximal simplex of $\pol$, then $\sigma$ is an isolated point of $|\pol|$. A \textit{subcomplex $\Tt$ of} $\pol$ is a finite simplicial complex such that $\Tt\subset\pol$. For each vertex $\nu\in\pol_{\bullet}$, the \textit{(open) star $\st(\nu,\pol)$ of $\nu$} in $\pol$ is the open neighbourhood $\bigcup_{\sigma\in\pol,\nu\in\sigma} \ov{\sigma}$. The \textit{closed star of $\nu$} in $\pol$ is defined as $\ol{\st}(\nu,\pol):=\cl(\st(\nu,\pol))$. Observe that the set of simplicies $\sigma\in\pol$ such that $\sigma\in\ol{\st}(\nu,\pol)$ is a subcomplex of $\pol$ for each vertex $\nu\in\pol_{\bullet}$. A \textit{refinement} (also called \textit{subdivision}) \textit{of} $\pol$ is a (finite) simplicial complex $\pol^*$ such that
\begin{itemize}
\item $|\pol^*|=|\pol|$,
\item each simplex $\sigma^*\in\pol^*$ is contained in some simplex $\sigma\in\pol$.
\end{itemize}
We denote $\sd^{\kappa}(\pol):=\sd(\sd^{\kappa-1}(\pol))$ the \textit{$k^{\text{th}}$ barycentric subdivision of} $\pol$, where $\sd^0(\pol):=\pol$. 

Let $\pol$ be a finite simplicial complex of $\R^n$ and $\Ll$ a finite simplicial complex of $\R^m$. A map $h:|\pol|\to\R^m$ is \textit{piecewise affine} if its restriction to each simplex $\sigma$ of $\pol$ is the restriction to $\sigma$ of an affine map $\R^n\to\R^m$. Note that this definition depends on the simplicial complex $\pol$ and not only on the polyhedron $|\pol|$. A map $h:|\pol_{\bullet}|\to \R^m$ extends uniquely to a piecewise affine map from $|\pol|$ to $\R^m$ and we denote this extension again $h:|\pol|\to\R^m$. In particular, a piecewise affine map is completely determined by its values on $|\pol_{\bullet}|$. A piecewise affine map $h:|\pol|\to |\Ll|$ is a \textit{simplicial map} if for each $\nu_0,\ldots,\nu_k\in\pol_{\bullet}$ that span a simplex of $\pol$, the images $h(\nu_0),\ldots,h(\nu_k)$ are vertices of $\Ll$ that span a simplex in $\Ll$.


\subsection{Paw\l{}ucki's desingularization}\label{pawu}

We recall here some definitions and the main result of \cite{p1} that we need in the following sections. Let $p\geq 0$ be an integer and $X\subset\R^n$ a compact definable set. A \textit{definable $\Cont^p$-triangulation of $X$} is a pair $(\pol,\varphi)$, where $\pol$ is a finite simplicial complex of $\R^n$ and $\varphi:|\pol|\to X$ is a definable homeomorphism such that for each $\sigma\in \pol$ the restriction $\varphi|_{\ov{\sigma}}$ is an embedding of class $\Cont^p$ of $\ov{\sigma}$ into $\R^n$. A definable $\Cont^0$-triangulation $(\pol,\varphi)$ is simply called \textit{definable triangulation}. If $\Ee$ is a finite family of definable subsets of $X$, we say that a definable triangulation $(\pol,\varphi)$ is \textit{compatible with $\Ee$} if for each $E\in \Ee$ the inverse image $\varphi^{-1}(E)$  is a union of some open simplices $\ov{\sigma}$, where $\sigma\in \pol$. A definable $\Cont^p$-triangulation $(\pol, \varphi)$ is a \textit{strict definable $\Cont^p$-triangulation} if the definable homeomorphism $\varphi:|\pol|\to X$ is of class $\Cont^p$.

In \cite{p1} Paw\l{}ucki showed the following remarkable result about the existence of a strict definable $\Cont^p$-triangulation $(\pol,\varphi)$ of $X$ compatible with a finite family of definable subsets $\Ee$ of $X$ that `desingularize' a given continuous definable map $f:X\to \R^m$ (in the sense that it smooths the map $f$ to the class $\Cont^p$, that is $f\circ\varphi$ is of class $\Cont^p$). Moreover, such triangulation can be chosen to be a refinement of a given definable triangulation of $X$ compatible with the finite family $\Ee$.

\begin{thm}[Strict $\Cont^p$-refinement theorem, {\cite[\S1]{p1}}]\label{pawudesing}
Let $X\subset\R^n$ be a compact definable set and $f:X\to \R^m$ a continuous definable map. Let $(\Tt,\psi)$ be a definable triangulation compatible with a finite family $\Ee$ of definable subsets of $X$. Then, for each integer $p\geq 1$, there exists a strict $\Cont^p$-triangulation $(\pol,\varphi)$ of $|\Tt|$ such that:
\begin{itemize}
\item $\pol$ is a refinement of $\Tt$.
\item $\varphi(\sigma)=\sigma$ for each simplex $\sigma\in\Tt$.
\item $(\pol,\psi\circ\varphi)$ is a strict $\Cont^p$-triangulation of $X$ compatible with the family $\Ee$.
\item $f\circ\psi\circ\varphi$ is of class $\Cont^p$.
\end{itemize}
\end{thm}


\section{Surjective simplicial approximation for definable maps.}

The purpose of this section is to show Proposition \ref{key3}. Let $\pol$ be a finite simplicial complex of $\R^n$, $\Ll$ a finite simplicial complex of $\R^m$ and $f:|\pol|\to|\Ll|$ a continuous map. A simplicial map $h:|\pol|\to|\Ll|$ is called a \textit{simplicial approximation of} $f$ if $f(\st(\nu,\pol))\subset \st(h(\nu),\Ll)$ for each vertex $\nu\in\pol_{\bullet}$. If $h$ is a simplicial approximation of $f$, then for each $x\in |\pol|$ there exists a simplex $\tau_x\in \Ll$ such that $f(x)\in\ov{\tau}_x$ and $h(x)\in\tau_x$ (see \cite[Lem.14.2]{mu}), so
$$
|f(x)-h(x)|_m<\diam(\tau_x)\leq \max_{\tau\in\Ll}\{\diam(\tau)\}.
$$

We now recall the finite simplicial approximation theorem. The proof is included here for the reader's convenience, as we will make use of it in Example \ref{notsurj}(ii) and Theorem \ref{simpapprox}.

\begin{thm}[Finite simplicial approximation, {\cite[Thm.16.1]{mu}}] \label{finitethmapprox}
Let $f:|\pol|\to|\Ll|$ be a continuous map. Then, there exists an integer $\kappa\geq 0$ such that $f$ has a simplicial approximation $h:|\sd^{\kappa}(\pol)|=|\pol|\to|\Ll|$.
\end{thm}
\begin{proof}
The family $\Ff:=\{f^{-1}(\st(\omega,\Ll))\}_{\omega\in\Ll_{\bullet}}$ is an open covering of $|\pol|$. As $|\pol|$ is compact, there exists a Lebesgue number $\lambda$ for the open covering $\Ff$. That is, a strictly positive number $\lambda$ with the following property: \textit{If $A\subset |\pol|$ satisfies $\diam(A):=\sup_{x,y\in A}|x-y|_n<\lambda$, then there exists $F\in \Ff$ such that $A\subset F$.} Let $\kappa\geq 0$ be an integer such that $\diam(\sigma)<\tfrac{\lambda}{2}$ for each $\sigma\in\sd^{\kappa}(\pol)$ \cite[Thm.15.4]{mu}. Then, for each vertex $\nu\in\sd^{\kappa}(\pol)_{\bullet}$ it holds $\diam(\st(\nu,\sd^{\kappa}(\pol))<\lambda$. In particular, there exists a vertex $\omega\in\Ll$ such that $f(\st(\nu,\sd^{\kappa}(\pol))\subset\st(\omega,\Ll)$. Thus, there exists a map $h:\sd^{\kappa}(\pol)_{\bullet}\to \Ll_{\bullet}$ such that $f(\st(\nu,\sd^{\kappa}(\pol))\subset \st(h(\nu),\Ll)$ for each vertex $\nu\in\pol$. By \cite[Lem.1.14]{mu}, the map $h$ extends to a simplicial map $|\sd^{\kappa}(\pol)|=|\pol|\to|\Ll|$, which is a simplicial approximation of $f$.
\end{proof}

As an immediate consequence:

\begin{cor}\label{finitesimp2}
Let $\veps>0$. Then, there exist two integers $\kappa,\ell\geq 0$ and a simplicial map $h:|\sd^{\kappa}(\pol)|=|\pol|\to|\sd^{\ell}(\Ll)|$ such that $\|f-h\|<\veps$.
\end{cor}
\begin{proof}
By \cite[Thm.15.4]{mu}, there exists an integer $\ell\geq 0$ such that $\diam(\tau)<\veps$ for each $\tau\in \sd^{\ell}(\Ll)$. Now, it is enough to apply the finite simplicial approximation theorem to the map $f:|\pol|\to|\sd^{\ell}(\Ll)|=|\Ll|$ in order to obtain the desired simplicial approximation $h:|\sd^{\kappa}(\pol)|\to|\sd^{\ell}(\Ll)|$.
\end{proof}

The following examples show that the finite simplicial approximation theorem does not guarantee, in general, that the simplicial approximation $h$ is surjective when the involved continuous map $f$ is surjective.

\begin{examples}\label{notsurj}
Let $\Ll$ be the simplicial complex whose underlying polyhedron $|\Ll|\subset\R^2$ is the one showed in Figure \ref{notexample}. 

(i) Let $\pol:=\{\{0\},\{1\},[0,1]\}$ and let $f:|\pol|=[0,1]\to|\Ll|$ be a surjective continuous map. For instance, one can consider as $f$ the Peano curve (see \cite{pe}). For each integer $\kappa\geq 0$ a simplicial map $h:|\sd^{\kappa}(\pol)|=[0,1]\to|\Ll|$ cannot be surjective, because the interval $|\pol|=[0,1]$ has dimension 1 and $|\Ll|$ has dimension 2. In particular, $f$ does not have a simplicial approximation which is surjective.

(ii) Let $\sigma\subset\R^2$ be a simplex of dimension 2 with vertices $\nu_0,\nu_1$ and $\nu_2$. Let $f:\sigma\to|\Ll|$ be an homeomorphism such that $f(\nu_i)=\omega_i$ for $i=0,1,2$. Thus, by the proof of Theorem \ref{finitethmapprox}, the map $f$ admits a simplicial approximation $h:\sigma\to|\Ll|$. As $h$ is a simplicial map, we have $h(\nu_i)\in \{\omega_0, \omega_2\}$ for each vertex $\nu_i$ of $\sigma$. In particular, the simplicial map $h$ is not surjective. $\sqbullet$

\begin{figure}[!ht]
\begin{center}
\begin{tikzpicture}[scale=1.2]

\draw[fill=blue!25,opacity=0.4] (0,0)--(0,2)--(2,2)--(2,0)--cycle;
\draw(0,0)--(0,2)--(2,2)--(2,0)--cycle;
\draw (0,0)--(2,2);
\draw (0,0) node[below left] {\small{$\omega_0$}};
\filldraw(0,0) circle (2pt);
\draw (2,0) node[below right] {\small{$\omega_1$}};
\filldraw(2,0) circle (2pt);
\draw (0,2) node[above left] {\small{$\omega_3$}};
\filldraw(0,2) circle (2pt);
\draw (2,2) node[above right] {\small{$\omega_2$}};
\filldraw(2,2) circle (2pt);
\end{tikzpicture}
\end{center}
\caption{\small{The polyhedron $|\Ll|\subset\R^2$.}}
\label{notexample}
\end{figure}
\end{examples}

Let $\Ll^*\subset\Ll$ be a subcomplex. The \textit{(open) star of $\Ll^*$} in $\Ll$ is the open neighbourhood of $|\Ll^*|$ in $|\Ll|$ defined as (see Figure \ref{figstar})
\begin{equation}\label{starsimp}
\st(\Ll^*,\Ll):=\bigcup_{\omega\in\Ll^*_{\bullet}}\st(\omega,\Ll)\subset |\Ll|.
\end{equation}
The \textit{closed star of $\Ll^*$} in $\Ll$ is defined as $\ol{\st}(\Ll^*,\Ll):=\cl(\st(\Ll^*,\Ll))$. Let now $\omega$ be a vertex of $\Ll$. The \textit{second (open) star $\st^{(2)}(\omega,\Ll)$ of} $\omega$ in $\Ll$ is defined as (see Figure \ref{figstar})
$$
\st^{(2)}(\omega,\Ll):=\st(\ol{\st}(\omega,\Ll),\Ll)\subset |\Ll|.
$$
In particular, as $\ol{\st}(\omega,\Ll)$ is a neighbourhood of $\omega$ in $|\Ll|$, then $\st^{(2)}(\omega,\Ll)$ is an open neighbourhood of $\omega$ in $\Ll$.
 
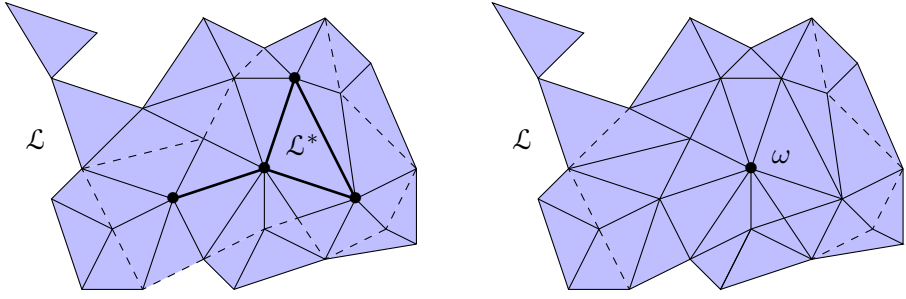
\begin{figure}[!ht]
\begin{center}
\begin{tikzpicture}[scale=0.8]


\draw[fill=blue!25,opacity=0.15,draw=none] (1.5,0.5)--(2.5,0.5)--(3.5,1)--(4,0.5)--(5.5,1)--(7,1.25)--(7,2.5)--(6.25,4.25)--(5.5,5)--(4.5,4.5)--(3.5,5)--(2.5,3.5)--(1,4)--(1.5,2.5)--(1,2)--cycle;

\draw[fill=white,draw=none](1.5,2.5)--(2.5,0.5)--(4.5,1.5)--(5.045,1.681)--(5.5,1)--(6.5,1.5)--(6.5,1.5)--(7,2.5)--(5.75,3.75)--(5.5,5)--(4.5,4.5)--(4,4)--(3.5,3)--cycle;

\draw[fill=blue!25,opacity=0.4,draw=none](1.5,2.5)--(2.5,0.5)--(4.5,1.5)--(5.045,1.681)--(5.5,1)--(6.5,1.5)--(6.5,1.5)--(7,2.5)--(5.75,3.75)--(5.5,5)--(4.5,4.5)--(4,4)--(3.5,3)--cycle;

\draw[fill=blue!25,opacity=0.15,draw=none] (1.75,4.75)--(0.25,5.25)--(1,4)--cycle;

\draw[fill=blue!25,opacity=0.55,draw=none](4.5,2.5)--(6,2)--(5,4)--cycle;

\draw (1.5,2.5) -- (1,2) -- (1.5,0.5) -- (2.5,0.5);
\draw  (1.5,0.5)  -- (2,1.5) -- (1,2);
\draw[dashed] (2,1.5) -- (2.5,0.5);
\draw[dashed] (2,1.5) -- (1.5,2.5);
\draw (1.5,2.5)--(3,2);
\draw (3,2) -- (2,1.5);
\draw (3,2) -- (2.5,0.5);
\draw[dashed] (1.5,2.5) -- (3.5,3);
\draw  (3.5,3) -- (3,2);
\draw (3,2) -- (3.5,1);
\draw[dashed] (3.5,1) -- (2.5,0.5);
\draw[line width=1pt]  (3,2) -- (4.5, 2.5);
\draw (4.5, 2.5)--(3.5,1);
\draw (4.5, 2.5)--(4.5,1.5);
\draw (3.5,1)--(4,0.5);
\draw[dashed] (3.5,1)--(4.5,1.5);
\draw (4.5,2.5)--(3.5,3);
\draw(1.5,2.5)--(2.5,3.5)--(3.5,3);
\draw(2.5,3.5)--(4,4)--(4.5,2.5);
\draw(1.5,2.5)--(1,4)--(2.5,3.5);
\draw[dashed] (3.5,3)--(4,4);
\draw[dashed] (4.5,1.5)--(5.045,1.681);
\draw (5.045,1.681)--(6,2);
\draw[line width=1pt](4.5,2.5)--(5,4)--(6,2)--cycle;
\draw(4,4)--(5,4);
\draw (5.5,1)--(6,2);
\draw(4.5,2.5)--(5.045,1.681);
\draw[dashed] (5.045,1.681)--(5.5,1);
\draw(4.5,1.5)--(5.5,1);
\draw(4.5,1.5)--(4,0.5);
\draw(4,0.5)--(5.5,1);
\draw(6,2)--(5.75,3.75)--(5,4);
\draw(6,2)--(6.5,1.5);
\draw[dashed] (6.5,1.5)--(5.5,1);
\draw(5.5,1)--(7,1.25);
\draw(7,1.25)--(6.5,1.5);
\draw(7,1.25)--(7,2.5);
\draw[dashed] (7,2.5)--(6.5,1.5);
\draw(7,2.5)--(6,2);
\draw[dashed](7,2.5)--(5.75,3.75);
\draw(5,4)--(5.5,5);
\draw[dashed] (5.5,5)--(5.75,3.75);
\draw[dashed] (4,4)--(4.5,4.5);
\draw (4.5,4.5)--(5,4);
\draw(4.5,4.5)--(5.5,5);
\draw(4,4)--(3.5,5)--(4.5,4.5);
\draw(2.5,3.5)--(3.5,5);
\draw(7,2.5)--(6.25,4.25)--(5.75,3.75);
\draw(6.25,4.25)--(5.5,5);
\draw(1.75,4.75)--(0.25,5.25)--(1,4)--cycle;

\draw (5.13,2.9) node{$\Ll^*$};
\draw (4.5,2.5) node{$\bullet$};
\draw (3,2) node{$\bullet$};
\draw (5,4) node{$\bullet$};
\draw (6,2) node{$\bullet$};
\draw (0.75,3) node{$\Ll$};

\draw (8.75,3) node{$\Ll$};

\draw[fill=blue!25,opacity=0.15,draw=none] (9.5,0.5)--(10.5,0.5)--(11.5,1)--(12,0.5)--(13.5,1)--(15,1.25)--(15,2.5)--(14.25,4.25)--(13.5,5)--(12.5,4.5)--(11.5,5)--(10.5,3.5)--(9,4)--(9.5,2.5)--(9,2)--cycle;

\draw[fill=blue!25,opacity=0.15,draw=none](9.75,4.75)--(8.25,5.25)--(9,4)--cycle;

\draw[fill=white,draw=none](9.5,2.5)--(10.5,0.5)--(11.5,1)--(12,0.5)--(13.5,1)--(14.5,1.5)--(15,2.5)--(13.75,3.75)--(13.5,5)--(12.5,4.5)--(11.5,5)--(10.5,3.5)--cycle;

\draw[fill=blue!25,opacity=0.4,draw=none](9.5,2.5)--(10.5,0.5)--(11.5,1)--(12,0.5)--(13.5,1)--(14.5,1.5)--(15,2.5)--(13.75,3.75)--(13.5,5)--(12.5,4.5)--(11.5,5)--(10.5,3.5)--cycle;

\draw (9.5,2.5) -- (9,2) -- (9.5,0.5) -- (10.5,0.5);
\draw  (9.5,0.5)  -- (10,1.5) -- (9,2);
\draw[dashed] (10,1.5) -- (10.5,0.5);
\draw[dashed] (10,1.5) -- (9.5,2.5);
\draw (9.5,2.5)--(11,2);
\draw (11,2) -- (10,1.5);
\draw (11,2) -- (10.5,0.5);
\draw (9.5,2.5) -- (11.5,3);
\draw  (11.5,3) -- (11,2);
\draw (11,2) -- (11.5,1) -- (10.5,0.5);
\draw(11,2) -- (12.5, 2.5);
\draw (12.5, 2.5)--(11.5,1);
\draw (12.5, 2.5)--(12.5,1.5)--(11.5,1);
\draw (11.5,1)--(12,0.5)--(12.5,1.5);
\draw (12.5,2.5)--(11.5,3);
\draw[dashed](9.5,2.5)--(10.5,3.5);
\draw (10.5,3.5)--(11.5,3);
\draw(10.5,3.5)--(12,4)--(12.5,2.5);
\draw(9.5,2.5)--(9,4)--(10.5,3.5);
\draw(11.5,3)--(12,4);
\draw(12.5,2.5)--(13,4)--(14,2)--cycle;
\draw(12,4)--(13,4);
\draw(12.5,1.5)--(14,2);
\draw(12.5,2.5)--(13.5,1)--(14,2);
\draw(12.5,1.5)--(13.5,1);
\draw(12.5,1.5)--(12,0.5);
\draw(12,0.5)--(13.5,1);
\draw(14,2)--(13.75,3.75)--(13,4);
\draw(14,2)--(14.5,1.5);
\draw[dashed] (14.5,1.5) --(13.5,1);
\draw(13.5,1)--(15,1.25)--(14.5,1.5);
\draw(15,1.25)--(15,2.5);
\draw[dashed](15,2.5)--(14.5,1.5);
\draw(15,2.5)--(14,2);
\draw[dashed](15,2.5)--(13.75,3.75);
\draw(13,4)--(13.5,5);
\draw[dashed](13.5,5)--(13.75,3.75);
\draw(12,4)--(12.5,4.5)--(13,4);
\draw(12.5,4.5)--(13.5,5);
\draw(12,4)--(11.5,5)--(12.5,4.5);
\draw(10.5,3.5)--(11.5,5);
\draw(15,2.5)--(14.25,4.25)--(13.75,3.75);
\draw(14.25,4.25)--(13.5,5);
\draw(9.75,4.75)--(8.25,5.25)--(9,4)--cycle;

\draw (12.5,2.5) node{{$\bullet$}};
\draw (13,2.7) node{$\omega$};
\end{tikzpicture}
\end{center}
\caption{\small{The (open) star of a subcomplex $\Ll^*\subset\Ll$ (left) and the second (open) star of a vertex $\omega\in\Ll$ (right).}}
\label{figstar}
\end{figure}

As mentioned in the introduction, in order to prove Theorem \ref{main} we employ Proposition \ref{key3}, that is a `surjective version' of Corollary \ref{finitesimp2} for surjective continuous definable maps. We begin with the following variant of the finite simplicial approximation theorem for continuous surjective definable maps.

\begin{prop}\label{simpapprox}
Let $f:|\pol|\to|\Ll|$ be a surjective continuous definable map. Then, there exists an integer $\kappa\geq 0$ and a surjective simplicial map $h:|\sd^{\kappa}(\pol)|=|\pol|\to|\Ll|$ such that 
$$
f(\st(\nu,\sd^{\kappa}(\pol)))\subset\st^{(2)}(h(\nu),\Ll)
$$
for each vertex $\nu\in\sd^{\kappa}(\pol)_{\bullet}$.
\end{prop}
\begin{proof}
The proof is carried out in several steps.

\noindent{\sc Step 1.} Let $\pol^*$ be any refinement of $\pol$, $\nu$ a vertex of $\pol^*$ and $\Lambda_{\nu}$ the set 
$$
\Lambda_{\nu}:=\{\sigma\in  \pol : \nu\in \sigma\}\subset\pol.
$$
Observe that $\Lambda_{\nu}\neq\varnothing$ for each vertex $\nu\in\pol_{\bullet}^*$. Let $\sigma_{\nu}\in \Lambda_{\nu}$ be a simplex of minimal dimension among the simplices of $\Lambda_{\nu}$, that is $\dim(\sigma_\nu)\leq \dim(\sigma)$ for each $\sigma\in \Lambda_{\nu}$. 
If $\nu\in \partial\sigma_{\nu}$, then there exists a proper face $\sigma$ of $\sigma_{\nu}$ such that $\nu\in \sigma$, so $\sigma\in\Lambda_{\nu}$. As $\sigma$ is a proper face of $\sigma_{\nu}$, then $\dim(\sigma)<\dim(\sigma_{\nu})$, which is a contradiction because $\sigma_{\nu}$ has minimal dimension among the simplices of $\Lambda_{\nu}$. In particular, $\nu\in \ov{\sigma}_{\nu}$. Let $\sigma\in \Lambda_{\nu}$ be another simplex of minimal dimension among the simplices of $\Lambda_{\nu}$, then $\nu\in \ov{\sigma}\cap \ov{\sigma}_{\nu}\neq\varnothing$, so $\sigma=\sigma_{\nu}$. We deduce: \textit{For each vertex $\nu\in\pol^*_{\bullet}$ there exists a unique simplex $\sigma_{\nu}\in \pol$ of minimal dimension such that $\nu\in\sigma_{\nu}$.}

Let $\omega_1,\ldots\,\omega_p$ be all the vertices of $\Ll$. The relation `$\omega_i\leq_{\Ll}\omega_j$ \textit{if and only if} $i\leq j$', defines a linear ordering on the set $\Ll_{\bullet}=\{\omega_1,\ldots,\omega_p\}$ of vertices of $\Ll$. For each vertex $\nu\in\pol^*_{\bullet}$ let $\sigma_{\nu}\in\pol$ be the (unique) simplex of minimal dimension such that $\nu\in\sigma_\nu$ and let $\nu_0,\ldots,\nu_{r_{\nu}}$ be the vertices of $\sigma_{\nu}$. To a simplicial map $h:|\pol|\to|\Ll|$ we associate the map $h^*:|\pol^*_{\bullet}|\to|\Ll_{\bullet}|$, defined as (see Figure \ref{accastar})
$$
h^*(\nu):=\min\{h(\nu_0),\ldots,h(\nu_{r_{\nu}})\},
$$
where the minimum is taken relatively to the linearly ordered set $(\Ll_{\bullet},\leq_{\Ll})$. Observe that the map $h^*$ is well-defined. In fact, for each $\nu\in\pol^*_{\bullet}$ there exists a unique simplex $\sigma_{\nu}\in\pol$ of minimal dimension such that $\nu\in\sigma_{\nu}$, so the set $\{h(\nu_0),\ldots,h(\nu_{r_{\nu}})\}$ is completely determined by $\nu$. Moreover, as $h:|\pol|\to|\Ll|$ is a simplicial map and $\nu_1,\ldots,\nu_{r_{\nu}}\in \pol_{\bullet}$ (because they are vertices of a simplex in $\pol$), then $\{h(\nu_0),\ldots,h(\nu_{r_{\nu}})\}\subset \Ll_{\bullet}$. We consider the piecewise affine extension $|\pol^*|=|\pol|\to\R^m$ of the map $h^*:|\pol^*_{\bullet}|\to|\Ll_{\bullet}|$ and we denote it again with $h^*$. 

\begin{figure}[ht!]
\centering
\begin{subfigure}[b]{0.45\textwidth}
\centering
\begin{tikzpicture}[scale=0.9]
    \path
      (90:3cm) coordinate (a) 
      (210:3cm) coordinate (b) 
      (-30:3cm) coordinate (c) ;
      
             \foreach \x  in {a,b,c}  {
 \foreach \y in {a,b,c}{
 \foreach \z in {a,b,c} \draw
      (barycentric cs:\x=1,\y=1,\z=1) --
      (barycentric cs:\x=1,\y=0,\z=0);
      }
      }
      
        \foreach \x  in {a,b,c}  {
 \foreach \y in {a,b,c}{
 \foreach \z in {a,b,c} \draw
      (barycentric cs:\x=1,\y=1,\z=1) --
      (barycentric cs:\x=1,\y=1,\z=0);
      }
      }
      
      \foreach \x  in {a,b,c}  {
 \foreach \y in {a,b,c}{
 \foreach \z in {a,b,c} \draw
      (barycentric cs:\x=11,\y=5,\z=2) --
      (barycentric cs:\x=1,\y=0,\z=0);
      }
      }
      
            \foreach \x  in {a,b,c}  {
 \foreach \y in {a,b,c}{
 \foreach \z in {a,b,c} \draw
      (barycentric cs:\x=11,\y=5,\z=2) --
      (barycentric cs:\x=4,\y=1,\z=1);
      }
      }

        \foreach \x  in {a,b,c}  {
 \foreach \y in {a,b,c}{
 \foreach \z in {a,b,c} \draw
      (barycentric cs:\x=11,\y=5,\z=2) --
      (barycentric cs:\x=1,\y=1,\z=1);
      }
      }
      
          \foreach \x  in {a,b,c}  {
 \foreach \y in {a,b,c}{
 \foreach \z in {a,b,c} \draw
      (barycentric cs:\x=11,\y=5,\z=2) --
      (barycentric cs:\x=3,\y=1,\z=0);
      }
      }
      
            \foreach \x  in {a,b,c}  {
 \foreach \y in {a,b,c}{
 \foreach \z in {a,b,c} \draw
      (barycentric cs:\x=11,\y=5,\z=2) --
      (barycentric cs:\x=1,\y=1,\z=0);
      }
      }
      
 \foreach \x  in {a,b,c}  {
 \foreach \y in {a,b,c}{
 \foreach \z in {a,b,c}
       \draw (barycentric cs:\x=11,\y=5,\z=2) --  (barycentric cs:\x=5,\y=5,\z=2);
      }
      }

\draw (a)--(b)--(c) --cycle;

\filldraw[blue] (barycentric cs:a=1,b=0,c=0) circle (3pt);
\filldraw[green] (barycentric cs:a=0,b=1,c=0) circle (3pt);
\filldraw[red](barycentric cs:a=0,b=0,c=1) circle (3pt);
\end{tikzpicture}
\end{subfigure}
\begin{subfigure}[b]{0.45\textwidth}
\centering
\begin{tikzpicture}[scale=0.9]
\path
      (90:3 cm) coordinate (a) 
      (210:3 cm) coordinate (b) 
      (-30:3 cm) coordinate (c);  
      
\draw (a)--(b)--(c) --cycle;
      
\foreach \x  in {a,b,c}  {
\foreach \y in {a,b,c}{
\foreach \z in {a,b,c} \draw
      (barycentric cs:\x=1,\y=1,\z=1) --
      (barycentric cs:\x=1,\y=0,\z=0);
      }
      }
      
\foreach \x  in {a,b,c}  {
\foreach \y in {a,b,c}{
\foreach \z in {a,b,c} \draw
      (barycentric cs:\x=1,\y=1,\z=1) --
      (barycentric cs:\x=1,\y=1,\z=0);
      }
      }
      
\foreach \x  in {a,b,c}{
\foreach \y in {a,b,c}{
\foreach \z in {a,b,c} \draw
      (barycentric cs:\x=11,\y=5,\z=2) --
      (barycentric cs:\x=1,\y=0,\z=0);
      }
      }
      
            \foreach \x  in {a,b,c}  {
 \foreach \y in {a,b,c}{
 \foreach \z in {a,b,c} \draw
      (barycentric cs:\x=11,\y=5,\z=2) --
      (barycentric cs:\x=4,\y=1,\z=1);
      }
      }

        \foreach \x  in {a,b,c}  {
 \foreach \y in {a,b,c}{
 \foreach \z in {a,b,c} \draw
      (barycentric cs:\x=11,\y=5,\z=2) --
      (barycentric cs:\x=1,\y=1,\z=1);
      }
      }
      
          \foreach \x  in {a,b,c}  {
 \foreach \y in {a,b,c}{
 \foreach \z in {a,b,c} \draw
      (barycentric cs:\x=11,\y=5,\z=2) --
      (barycentric cs:\x=3,\y=1,\z=0);
      }
      }
      
            \foreach \x  in {a,b,c}  {
 \foreach \y in {a,b,c}{
 \foreach \z in {a,b,c} \draw
      (barycentric cs:\x=11,\y=5,\z=2) --
      (barycentric cs:\x=1,\y=1,\z=0);
      }
      }
      
 \foreach \x  in {a,b,c}  {
 \foreach \y in {a,b,c}{
 \foreach \z in {a,b,c}
       \draw (barycentric cs:\x=11,\y=5,\z=2) --  (barycentric cs:\x=5,\y=5,\z=2);
      }
      }
      
\filldraw[blue] (barycentric cs:a=1,b=0,c=0) circle (3pt);
\filldraw[green] (barycentric cs:a=0,b=1,c=0) circle (3pt);
\filldraw[red](barycentric cs:a=0,b=0,c=1) circle (3pt);

\filldraw[red](barycentric cs:a=1,b=1,c=1) circle (3pt);

\filldraw[red](barycentric cs:a=0,b=1,c=1) circle (3pt);
\filldraw[red](barycentric cs:a=0,b=1,c=3) circle (3pt);
\filldraw[red](barycentric cs:a=0,b=3,c=1) circle (3pt);

\filldraw[red](barycentric cs:a=1,b=0,c=1) circle (3pt);
\filldraw[red](barycentric cs:a=3,b=0,c=1) circle (3pt);
\filldraw[red](barycentric cs:a=1,b=0,c=3) circle (3pt); 

\filldraw[blue](barycentric cs:a=1,b=1,c=0) circle (3pt);
\filldraw[blue](barycentric cs:a=3,b=1,c=0) circle (3pt);
\filldraw[blue](barycentric cs:a=1,b=3,c=0) circle (3pt); 

 \filldraw[red] (barycentric cs:a=5,b=5,c=2) circle (3pt);
  \filldraw[red] (barycentric cs:a=5,b=2,c=5) circle (3pt);
   \filldraw[red] (barycentric cs:a=2,b=5,c=5) circle (3pt);

 \filldraw[red] (barycentric cs:a=5,b=11,c=2) circle (3pt);
  \filldraw[red] (barycentric cs:a=5,b=2,c=11) circle (3pt);
   \filldraw[red] (barycentric cs:a=2,b=5,c=11) circle (3pt);
    \filldraw[red] (barycentric cs:a=2,b=11,c=5) circle (3pt);
  \filldraw[red] (barycentric cs:a=11,b=2,c=5) circle (3pt);
   \filldraw[red] (barycentric cs:a=11,b=5,c=2) circle (3pt);
   
    \filldraw[red] (barycentric cs:a=4,b=1,c=1) circle (3pt);
  \filldraw[red] (barycentric cs:a=1,b=4,c=1) circle (3pt);
   \filldraw[red] (barycentric cs:a=1,b=1,c=4) circle (3pt);

\end{tikzpicture}
\end{subfigure}
\caption{\small{The values of a simplicial map $h$ on $\sigma_{\bullet}$ (left) and the corresponding values of $h^*$ on $\sd^2(\sigma)_{\bullet}$ (right). The red vertices are sent in $\omega_{i_1}$, the blue ones in $\omega_{i_2}$ and the green ones in $\omega_{i_3}$, where $\omega_{i_1}<_{\Ll}\omega_{i_2}<_{\Ll}\omega_{i_3}$.}}
\label{accastar}
\end{figure}
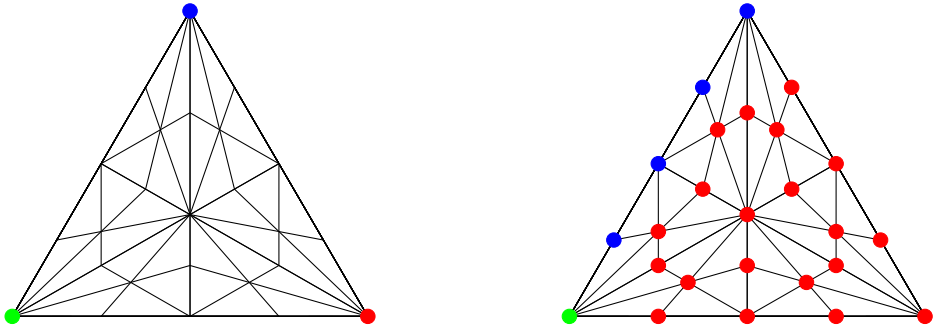

\noindent{\sc Step 2.} We show in this step: \textit{$h^*:|\pol^*|=|\pol|\to|\Ll|$ is a simplicial map such that $h^*(\sigma)=h(\sigma)$ for each simplex $\sigma\in\pol$. Moreover, if $h$ is a simplicial approximation of a continuous map $f:|\pol|\to|\Ll|$, then also $h^*$ is a simplicial approximation of $f$.}

First, we show: \textit{$h^*:|\pol^*|\to|\Ll|$ is a simplicial map}. Let $\sigma^*\in\pol^*$ and let $\nu^*_0,\ldots,\nu^*_{r^*}$ be its vertices. For each $j=0,\ldots,r^*$ let $\sigma_{\nu^*_j}\in\pol$ be the (unique) simplex of minimal dimension such that $\nu^*_j\in\sigma_{\nu^*_j}$. As $\pol^*$ is a refinement of $\pol$, there exists a simplex $\sigma\in\pol$ such that $\sigma^*\subset\sigma$. Moreover, $\sigma_{\nu^*_j}\subset\sigma$ for each $j=0,\ldots,r^*$. In fact, as $\nu_j^*\in\sigma^*\subset \sigma$ and $\nu_j^*\in\ov{\sigma}_{\nu_j^*}$ (see {\sc Step 1}), then $\nu_j^*\in\ov{\sigma}_{\nu_j^*}\cap \sigma$, so $\sigma_{\nu_j^*}$ is a face of $\sigma$. Let $\nu^j_0,\ldots,\nu^j_{r_j}\in\pol_{\bullet}$ be the vertices of $\sigma_{\nu^*_j}$. In particular, $\nu^j_0,\ldots,\nu^j_{r_j}\in\sigma$. As $h$ is a simplicial map and $\sigma$ is convex, then for each $j=0,\ldots,r^*$ 
\begin{multline*}
h^*(\nu^*_j)=\min\{h(\nu_0^j),\ldots,h(\nu_{r_j}^j)\}\in \Conv(\{h(\nu_0^j),\ldots,h(\nu_{r_j}^j)\})\\
=h(\Conv(\{\nu_0^j,\ldots,\nu_{r_j}^j\})\subset h(\sigma)\in\Ll
\end{multline*}
and $h^*(\nu^*_j)$ is a vertex of $h(\sigma)$, because $\nu^j_i$ is a vertex of $\pol$ for each $i=0,\ldots,r_j$. As $h^*$ is piecewise affine, we deduce
\begin{equation}\label{accastar2}
h^*(\sigma^*)=h^*(\Conv(\{\nu^*_0,\ldots,\nu^*_{r^*}\}))=\Conv(\{h^*(\nu^*_0),\ldots,h^*(\nu^*_{r^*})\})\subset h(\sigma).
\end{equation}
Thus, as $h(\sigma)\in\Ll$ and $h^*(\nu^*_j)$ is a vertex of $h(\sigma)$, we have that $h^*(\sigma^*)$ is a face of $h(\sigma)$, so $h^*(\sigma^*)\in\Ll$. We conclude that $h^*(|\pol^*|)\subset |\Ll|$ and $h^*:|\pol^*|=|\pol|\to|\Ll|$ is a simplicial map, as desired.

Then, we show: \textit{$h^*(\sigma)=h(\sigma)$ for each $\sigma\in\pol$}. Let $\sigma\in\pol$. As $\pol^*$ is a refinement of $\pol$, then 
\begin{equation}\label{bigcup}
\sigma=\bigcup_{\{\sigma^*\in\pol^* : \sigma^*\subset\sigma\}}\sigma^*.
\end{equation}
For each $\sigma^*\in\pol^*$ such that $\sigma^*\subset \sigma$, by \eqref{accastar2}, we have that $h^*(\sigma^*)\subset h(\sigma)$. In particular,
$$
h^*(\sigma)=h^*\Big(\bigcup_{\{\sigma^*\in\pol^* : \sigma^*\subset\sigma\}}\sigma^*\Big)=\bigcup_{\{\sigma^*\in\pol^* : \sigma^*\subset\sigma\}}h^*(\sigma^*)\subset h(\sigma).
$$
Thus, in order to conclude it is enough to show: \textit{There exists $\sigma^*\in\pol^*$ such that $\sigma^*\subset\sigma$ and $h^*(\sigma^*)=h(\sigma)$.} Let $\nu_0,\ldots,\nu_r$ be the vertices of $\sigma$ and $k:=\dim(h(\sigma))$. As $h$ is a simplicial map, up to reordering the vertices $\nu_0,\ldots,\nu_r$, we may assume that there exists $k\leq r$ such that $h(\sigma)=\Conv(\{h(\nu_0),\ldots,h(\nu_k)\})$ and $h(\nu_0)<_{\Ll}\ldots <_{\Ll} h(\nu_k)$. Let $\sigma_k\in\pol$ be the simplex of vertices $\nu_0,\ldots,\nu_k$. If there exists a simplex $\sigma^*_k\in\pol^*$ such that $\sigma_k^*\subset\sigma_k$ and $h^*(\sigma^*_k)=h(\sigma)$, then we are done, because $\sigma_k\subset\sigma$ and $h^*(\sigma)\subset h(\sigma)$. Thus, up to replacing $\sigma$ with $\sigma_k$, we may assume $\dim(\sigma)=\dim(h(\sigma))=k$. We proceed inductively on $k$. It is clear that if $k=0$, then $h^*(\sigma)=h(\sigma)$. Assume $k \geq 1$. Let $\sigma_0$ be the facet of $\sigma$ such that $\nu_0\not\in\sigma_0$ and let $\nu\in\pol^*_{\bullet}\cap\sigma$ be a vertex such that $\nu\not\in\sigma_0$. Let $\sigma_{\nu}\in\pol$ be the (unique) simplex of minimal dimension such that $\nu\in\sigma_{\nu}$. As $\nu\in\sigma$, proceeding as in {\sc Step 1}, we deduce that $\sigma_{\nu}$ is a face of $\sigma$. In particular, the vertices of $\sigma_{\nu}$ are vertices of $\sigma$. Moreover, as $\nu\not\in\sigma_0$, then $\nu_0$ is a vertex of $\sigma_{\nu}$, because otherwise, as $\sigma_{\nu}$ is a face of $\sigma$, we would have $\sigma_{\nu}\subset\sigma_0$ and $\nu\not\in\sigma_{\nu}$. We deduce,
\begin{equation}\label{nu0}
h(\nu_0)=\min\{h(\nu_0),\ldots,h(\nu_k)\} \leq_{\Ll} \min_{\nu_j\in\sigma_{\nu}}\{h(\nu_j)\}\leq h(\nu_0), 
\end{equation}
so $h^*(\nu)=\min_{\nu_j\in\sigma_{\nu}}\{h(\nu_j)\}=h(\nu_0)$ for each vertex $\nu\in\pol^*$ such that $\nu\in\sigma\setminus\sigma_0$. As $\dim(\sigma_0)=k-1$, by inductive hypothesis, there exists a simplex $\sigma^*_0\subset\sigma_0$ such that 
$$
h^*(\sigma^*_0)=h(\sigma_0)=\Conv(\{h(\nu_1),\ldots,h(\nu_k)\}).
$$
In particular, $\dim(\sigma^*_0)=k-1$. Let $\sigma^*\subset\sigma$ be the (unique) simplex of $\pol^*$ of dimension $k$ such that $\sigma^*_0\subset\sigma^*$ and let $\nu^*$ be the vertex of $\sigma^*$ such that $\nu^*\not\in\sigma_0$. Then, by \eqref{nu0}, $h^*(\nu^*)=h(\nu_0)$. As $h^*:|\pol^*|\to|\Ll|$ is a simplicial map and both $\sigma^*,\sigma^*_0\in\pol^*$, we deduce
\begin{multline*}
h^*(\sigma^*)=h^*(\Conv(\{\nu^*\}\cup\sigma_0^*))=\Conv(\{h^*(\nu^*)\}\cup h^*(\sigma_0^*))\\
=\Conv( \{h(\nu_0)\}\cup h(\sigma_0))=\Conv(\{h(\nu_0),h(\nu_1),\ldots,h(\nu_k)\})=h(\sigma),
\end{multline*}
as required.

Let $f:|\pol|\to|\Ll|$ be a continuous map. Next, we show: \textit{If $h$ is a simplicial approximation of $f$, then also $h^*$ is a simplicial approximation of $f$.} As $h^*:|\pol^*|\to|\Ll|$ is a simplicial map, we only need to show that $f(\st(\nu,\pol^*))\subset \st(h^*(\nu),\Ll)$ for each vertex $\nu\in\pol^*_{\bullet}$. Let $\nu$ be a vertex of $\pol^*$, $\sigma_\nu\in\pol$ the (unique) simplex of minimal dimension such that $\nu\in\sigma_\nu$ and $\nu_0,\ldots,\nu_r$ the vertices of $\sigma_{\nu}$. Let $k\in\{0,\ldots,r\}$ be such that $h^*(\nu)=h(\nu_k)$. As $\nu\in\sigma_{\nu}$ and $\pol^*$ is a refinement of $\pol$, then each simplex $\sigma\in\pol^*$ such that $\sigma\cap\st(\nu,\pol^*)\neq\varnothing$ satisfies $\sigma\cap \sigma_\nu\neq\varnothing$. Thus, $\sigma\cap\st(\nu,\pol^*)\subset \st(\nu_j,\pol)$ for each $j=0,\ldots,r$ (again because $\pol^*$ is a refinement of $\pol$). In particular, 
$$
\st(\nu,\pol^*)=\bigcup_{\{\sigma\in\pol^* : \sigma\cap\st(\nu,\pol^*)\neq \varnothing\}}(\sigma\cap\st(\nu,\pol^*))\subset\bigcap_{j=0}^{r} \st(\nu_j,\pol)\subset\st(\nu_k,\pol).
$$
As $h$ is a simplicial approximation of $f$, we deduce
$$
f(\st(\nu,\pol^*))\subset  f(\st(\nu_k,\pol))\subset\st(h(\nu_k),\Ll)=\st(h^*(\nu),\Ll),
$$
as desired.


\noindent{\sc Step 3.} Let $f:|\pol|\to|\Ll|$ be a surjective continuous definable map and $\tau$ a maximal simplex of $\Ll$. As $\tau$ is a maximal simplex of $\Ll$, its (relative) interior $\ov{\tau}$ is open in $|\Ll|$, so $f^{-1}(\ov{\tau})$ is open in $|\pol|$. Let $\Xi:=\Xi_{\tau}$ be the set of simplices $\sigma\in\pol$ such that $\ov{\sigma}\cap f^{-1}(\ov{\tau})\neq \varnothing$. As $f$ is surjective and
$$
|\pol|=\bigcup_{\sigma\in\pol}\sigma=\bigcup_{\sigma\in\pol}\ov{\sigma},
$$
we have $\Xi\neq\varnothing$. Let $\sigma_{\tau}\in\Xi$ be a simplex of maximal dimension among the simplices of $\Xi$, that is $\dim(\sigma)\leq \dim(\sigma_{\tau})$ for all $\sigma\in\Xi$. Assume that $\sigma_{\tau}$ is not a maximal simplex of $\pol$, then $\sigma_{\tau}$ is a proper face of another simplex $\sigma\in\pol$. In particular $\dim(\sigma_{\tau})<\dim(\sigma)$. As $f^{-1}(\ov{\tau})$ is open, then $\ov{\sigma}\cap f^{-1}(\ov{\tau})\neq\varnothing$, so $\sigma\in\Xi$. Thus, $\dim(\sigma_\tau)<\dim(\sigma)\leq\dim(\sigma_{\tau})$, which is a contradiction. We deduce that $\sigma_{\tau}$ is a maximal simplex of $\pol$.

As $f$ is a definable map and 
$$
\ov{\tau}=f\Big(\bigcup_{\sigma\in \Xi} \ov{\sigma}\cap f^{-1}(\ov{\tau})\Big)=\bigcup_{\sigma\in\Xi}f(\ov{\sigma}\cap f^{-1}(\ov{\tau})),
$$
we deduce (here is the only point where we use in an essential way that $f$ is definable, see also Example \ref{notsurj}(i)),
\begin{align*}
\dim(\tau)&=\dim(\ov{\tau})=\max_{\sigma\in\Xi}\{\dim(f(\ov{\sigma}\cap f^{-1}(\ov{\tau})))\}\leq\max_{\sigma\in\Xi}\{\dim(\ov{\sigma}\cap f^{-1}(\ov{\tau}))\}\\
&\leq \max_{\sigma\in\Xi}\{\dim(\ov{\sigma})\}=\max_{\sigma\in\Xi}\{\dim(\sigma)\}\leq\dim(\sigma_{\tau}).
\end{align*}

Let $\tau_1,\ldots,\tau_s$ be the maximal simplices of $\Ll$. We have: \textit{For each $k=1,\ldots,s$ there exists a maximal simplex $\sigma_k\in \pol$ such that $\dim(\sigma_k)\geq \dim(\tau_k)$ and $\ov{\sigma_k}\cap f^{-1}(\ov{\tau}_k)\neq\varnothing$.}

\noindent{\sc Step 4.} Let $\sigma_1,\ldots,\sigma_s$ be the maximal simplices of $\pol$ defined in {\sc Step 3}. It might happen that $\sigma_k=\sigma_{k'}$ for some $k,k'\in\{1,\ldots,s\}$ such that $k\neq k'$. We show in this step: \textit{There exists an integer $\kappa^{\circ}\geq 0$ with the following property: for each integer $\kappa^{\bullet}\geq \kappa^{\circ}$ and for each $k=1,\ldots,s$ there exists a maximal simplicex $\sigma'_k$ of $\sd^{\kappa^{\bullet}}(\pol)$ such that:
\begin{itemize}
\item $\dim(\sigma'_k)\geq \dim(\tau_k)$,
\item $\ov{\sigma}'_k\cap f^{-1}(\ov{\tau}_k)\neq\varnothing$,
\item if $k\neq k'$, then $\sigma'_k\neq \sigma'_k$.
\end{itemize}}

Let $\tau_1,\ldots,\tau_s$ be the maximal simplices of $\Ll$ and let $k\in\{1,\ldots,s\}$. As $\sigma_k$ is a maximal simplex of $\pol$, $\ov{\sigma}_k$ is open in $\pol$. In particular, also $\ov{\sigma_k}\cap f^{-1}(\ov{\tau}_k)$ is open in $\pol$, so in $\ov{\sigma}_k$. Thus, $\dim(\ov{\sigma_k}\cap f^{-1}(\ov{\tau}_k))=\dim(\ov{\sigma}_k)=\dim(\sigma_k)$ because $\ov{\sigma_k}\cap f^{-1}(\ov{\tau}_k)$ is open in $\ov{\sigma}_k$ and non-empty. Let $i_1,\ldots,i_{\ell_k}\in \{1,\ldots,s\}$ be all the indices such that $\sigma_{i_j}=\sigma_k$. We assume that the indices $i_j$ are all distinct. As 
$$
 f(\ov{\sigma}_{i_j}\cap f^{-1}(\ov{\tau}_{i_j}))\subset f(\ov{\sigma}_{i_j})\cap\ov{\tau}_{i_j}\subset \ov{\tau}_{i_j},
$$
and $\ov{\tau}_{i_j}\cap \ov{\tau}_{i_{j'}}=\varnothing$ if $j\neq j'$, we deduce that
\begin{equation}\label{diversi4}
(\ov{\sigma}_{i_j}\cap f^{-1}(\ov{\tau}_{i_j}))\cap(\ov{\sigma}_{i_{j'}}\cap f^{-1}(\ov{\tau}_{i_{j'}}))=(\ov{\sigma}_{k}\cap f^{-1}(\ov{\tau}_{i_j}))\cap(\ov{\sigma}_{k}\cap f^{-1}(\ov{\tau}_{i_{j'}}))=\varnothing
\end{equation}
for each $j\neq j'$.

As $\sigma_k$ is a convex subset of $\R^n$, by \cite[11.1.2.7]{be}, for each $j=1,\ldots,\ell_k$ there exist $x_{i_j}\in  \ov{\sigma}_{i_j}\cap f^{-1}(\ov{\tau}_{i_j})$ and $\eps_{i_j}>0$ such that
\begin{equation}\label{pallina}
\{x\in |\pol| : |x_{i_j}-x|_n<\eps_{i_j}\}=\{x\in \ov{\sigma}_k : |x_{i_j}-x|_n<\eps_{i_j}\} \subset\ov{\sigma}_{i_j}\cap f^{-1}(\ov{\tau}_{i_j}).
\end{equation}
Let
$$
\eps_{k}:=\min\{\eps_{i_1},\ldots,\eps_{i_{\ell_k}}\}>0 \text{\, and \, } \eps:=\min\{\eps_1,\ldots,\eps_s\}>0.
$$

By \cite[Thm.15.4]{mu}, there exists an integer $\kappa^{\circ}\geq 0$ such that $\diam(\sigma)<\eps$ for each $\sigma\in \sd^{\kappa^{\circ}}(\pol)$. Let $\kappa^{\bullet}$ be an integer such that $\kappa^{\bullet}\geq \kappa^{\circ}$. As $\sd^{\kappa^{\bullet}}(\pol)$ is a refinement of $\sd^{\kappa^{\circ}}(\pol)$, then $\diam(\sigma)<\eps$ for each $\sigma\in \sd^{\kappa^{\bullet}}(\pol)$. Let $j\in\{1,\ldots,\ell_k\}$. Then, there exists a simplex $\sigma'_{i_j}\in \sd^{\kappa^{\bullet}}(\pol)$ such that $\sigma'_{i_j}\subset \ov{\sigma}_{i_j}\cap f^{-1}(\ov{\tau}_{i_j})$ and $\dim(\sigma'_{i_j})=\dim(\sigma_{i_j})(=\dim(\sigma_k))$. In fact, as
$$
\sigma_k=\sigma_{i_j}=\bigcup\{\sigma\in \sd^{\kappa^{\bullet}}(\pol) : \sigma\subset\sigma_k \text{ and } \dim(\sigma)=\dim(\sigma_k)\},
$$
then there exists a simplex $\sigma'_{i_j}\in \sd^{\kappa^{\bullet}}(\pol)$ such that $\sigma'_{i_j}\subset\sigma_k$, $\dim(\sigma'_{i_j})=\dim(\sigma_k)$ and $x_{i_j}\in\sigma'_{i_j}$. As $\diam(\sigma'_{i_j})<\eps\leq\eps_k\leq \eps_{i_j}$, by \eqref{pallina}, we deduce
$$
\sigma'_{i_j}\subset \{x\in |\pol| : |x_{i_j}-x|_n<\eps_{i_j}\}\subset \ov{\sigma}_{i_j}\cap f^{-1}(\ov{\tau}_{i_j}),
$$ 
because for each $x\in \sigma'_{i_j}$ it holds
$$
|x_{i_j}-x|_n\leq\diam(\sigma'_{i_j})<\eps_{i_j}.
$$ 
In particular, $\ov{\sigma}'_{i_j}\cap f^{-1}(\ov{\tau}_{i_j})\neq\varnothing$. Moreover, by \eqref{diversi4}, the simplices $\sigma'_{i_j}$ are all distinct. In order to conclude, it is enough to observe that
$$
\dim(\sigma'_{i_j})=\dim(\sigma_k)=\dim(\sigma_{i_j})\geq \dim(\tau_{i_j})
$$
for each $j=1,\ldots,\ell_k$. 

\noindent{\sc Step 5.} By the finite simplicial approximation theorem (see Theorem \ref{finitethmapprox} and its proof), there exists an integer $\kappa^*\geq 0$ such that for each $\kappa^{\star}\geq \kappa^*$ there exists a simplicial approximation $|\sd^{\kappa^{\star}}(\pol)|=|\pol|\to |\Ll|$ of $f$. Let $\kappa^{\circ}\geq 0$ be the integer defined in {\sc Step 4}. We may assume that $\kappa^*\geq \kappa^{\circ}$. Let $h_0:|\pol^*|\to|\Ll|$ be a simplicial approximation of $f$. The simplicial map $h_0$ might not be surjective, so we have to modify it. Let $\kappa:=\kappa^*+2$ and let 
$$
\pol^*:=\sd^\kappa(\pol)=\sd^{\kappa^*+2}(\pol)=\sd^{2}(\sd^{\kappa^*}(\pol))
$$ 
be the second barycentric subdivision of $\sd^{\kappa^*}(\pol)$. In order to lighten the notation we replace $\pol$ with $\sd^{\kappa^*}(\pol)$, so $\pol^*=\sd^2(\pol)$. Let $h^*:|\pol^*|=|\pol|\to|\Ll|$ be the map associated to the simplicial map $h_0$ with respect to the refinement $\pol^*=\sd^2(\pol)$ (see {\sc Step 1}). By {\sc Step 2}, $h^*$ is a simplicial approximation of $f$ such that $h^*(\sigma)=h_0(\sigma)$ for each $\sigma\in\pol$. 

Let $\tau_k$ be a maximal simplex of $\Ll$ and $\sigma_k$ a maximal simplex of $\pol$ such that  (see {\sc Step 3}) $\dim(\sigma_k)\geq\dim(\tau_k)$ and $\ov{\sigma_k}\cap f^{-1}(\ov{\tau}_k)\neq\varnothing$. Let $\nu_0,\ldots,\nu_{r_k}$ be the vertices of $\sigma_k$. By {\sc Step 4}, we may assume that the simplices $\sigma_k$ are all distinct. As $h_0:|\pol|\to|\Ll|$ is a simplicial approximation of $f$, we have
$$
f(\ov{\sigma}_k\cap f^{-1}(\ov{\tau}_k))\subset f(\ov{\sigma}_k)\cap f(f^{-1}(\ov{\tau}_k))\subset f(\st(\nu_j,\pol))\cap \ov{\tau}_k\subset \st(h_0(\nu_j),\Ll)\cap \ov{\tau}_k.
$$
for each $j=0,\ldots,r_k$. In particular, $\st(h_0(\nu_j),\Ll)\cap \ov{\tau}_k\neq \varnothing$, so $h_0(\nu_j)$ is a vertex of $\tau_k$. We deduce that $h_0(\sigma_k)$ is a face of $\tau_k$. By {\sc Step 2}, $h^*(\sigma_k)=h_0(\sigma_k)$, so also $h^*(\sigma_k)$ is a face of $\tau_k$.

\noindent{\sc Step 6.} As $|\Ll|$ has finitely many connected components, we may assume that $|\Ll|$ is connected. If $\tau\in \Ll_{\bullet}$ is a vertex that is also a maximal simplex of $\Ll$, then $\tau$ is an isolated point of $|\Ll|$, so $|\Ll|=\tau$, because $|\Ll|$ is connected. In particular, as the thesis is clear when $|\Ll|$ is a single point, we may assume that: \textit{All the maximal simplices of $\Ll$ have dimension $\geq 1$}.

Let $\tau_1,\ldots,\tau_s$ be the maximal simplices of $\Ll$ and let $d_k:=\dim(\tau_k)\geq 1$ for each $k=1,\ldots,s$. Let $\sigma_1,\ldots,\sigma_s$ be maximal simplices of $\pol$ such that $r_k:=\dim(\sigma_k)\geq d_k$ and $\ov{\sigma_k}\cap f^{-1}(\ov{\tau}_k)\neq\varnothing$ for each $k=1,\ldots,s$ (see {\sc Step 3}). By {\sc Step 4}, we may assume that $\sigma_k\neq\sigma_{k'}$ for each $k\neq k'$. We want to construct: \textit{A simplicial map $h:|\pol^*|=|\pol|\to |\Ll|$ such that
\begin{itemize}
\item[{\rm(i)}] $h(\sigma_k)=\tau_k$ for each $k=1,\ldots,s$,
\item[{\rm(ii)}]  $h=h^*$ on $|\pol|\setminus(\ov{\sigma}_1\cup\ldots\cup\ov{\sigma}_s)$.
\end{itemize}
} 

For each $k=1,\ldots,s$ let $b_k$ be the barycentre of $\sigma_k$ and $\sigma'_k\in\pol^*=\sd^2(\pol)$ a simplex of dimension $r_k=\dim(\sigma_k)$ such that $b_k$ is a vertex of $\sigma'_k$. As $\pol^*$ is the second barycentric subdivision $\sd^2(\pol)$ of $\pol$ and $\sigma_k$ is a maximal simplex of $\pol$, it follows that $\sigma'_k\subset \ov{\sigma}_k$. In particular, the vertices of $\sigma_k'$ are not vertices of $\sigma_k$. As $\sigma_k\neq\sigma_{k'}$ if $k\neq k'$, then $\ov{\sigma}_k\cap \ov{\sigma}_{k'}=\varnothing$. Thus, 
\begin{equation}\label{distintisimp}
\sigma'_k\cap\sigma'_{k'}\subset\ov{\sigma}_k\cap\ov{\sigma}_{k'}=\varnothing
\end{equation}
for each $k\neq k'$. 

For each $k=1,\ldots,s$ let $\nu_{0,k}:=b_{k},\nu_{1,k},\ldots,\nu_{r_k,k}$ be the vertices of $\sigma'_k$ and $\omega_{0,k},\ldots,\omega_{d_k,k}$ the vertices of $\tau_k$. We define the map (see Figure \ref{accastar3})
$$
h:|\pol^*_{\bullet}|\to |\Ll_{\bullet}|,\, \nu\mapsto
\begin{cases}
\omega_{i,k}, \text{ if } \nu=\nu_{i,k} \text{ for some } k=0,\ldots,s \text{ and } i=0,\ldots,d_k;\\
\omega_{0,k}, \text{ if } \nu\in \{\nu_{d_k+1,k},\ldots,\nu_{r_k,k}\} \text{ for some } k=0,\ldots,s;\\
h^*(\nu), \text{ otherwise.}
\end{cases}
$$
By \eqref{distintisimp}, the map $h$ is well-defined. Observe that we have $h(\nu)=h^*(\nu)$ for each vertex $\nu\in \pol^*_{\bullet}\setminus (\sigma'_1\cup\ldots\cup\sigma'_k)$. We denote again with $h$ the piecewise affine extension $|\pol^*|=|\pol|\to\R^m$ of the map $h:|\pol^*_{\bullet}|\to |\Ll_{\bullet}|$. Let us show that $h$ is a simplicial map from $|\pol^*|$ to $|\Ll|$ which satisfies property (i) and (ii).

\begin{figure}[ht!]
\centering
\begin{subfigure}[b]{0.45\textwidth}
\centering
\begin{tikzpicture}[scale=0.9]
\path
      (90:3cm) coordinate (a) 
      (210:3cm) coordinate (b) 
      (-30:3cm) coordinate (c);  
      
\draw (a)--(b)--(c) --cycle;
      
\foreach \x  in {a,b,c}  {
\foreach \y in {a,b,c}{
\foreach \z in {a,b,c} \draw
      (barycentric cs:\x=1,\y=1,\z=1) --
      (barycentric cs:\x=1,\y=0,\z=0);
      }
      }
      
\foreach \x  in {a,b,c}  {
\foreach \y in {a,b,c}{
\foreach \z in {a,b,c} \draw
      (barycentric cs:\x=1,\y=1,\z=1) --
      (barycentric cs:\x=1,\y=1,\z=0);
      }
      }
      
\foreach \x  in {a,b,c}{
\foreach \y in {a,b,c}{
\foreach \z in {a,b,c} \draw
      (barycentric cs:\x=11,\y=5,\z=2) --
      (barycentric cs:\x=1,\y=0,\z=0);
      }
      }
      
            \foreach \x  in {a,b,c}  {
 \foreach \y in {a,b,c}{
 \foreach \z in {a,b,c} \draw
      (barycentric cs:\x=11,\y=5,\z=2) --
      (barycentric cs:\x=4,\y=1,\z=1);
      }
      }

        \foreach \x  in {a,b,c}  {
 \foreach \y in {a,b,c}{
 \foreach \z in {a,b,c} \draw
      (barycentric cs:\x=11,\y=5,\z=2) --
      (barycentric cs:\x=1,\y=1,\z=1);
      }
      }
      
          \foreach \x  in {a,b,c}  {
 \foreach \y in {a,b,c}{
 \foreach \z in {a,b,c} \draw
      (barycentric cs:\x=11,\y=5,\z=2) --
      (barycentric cs:\x=3,\y=1,\z=0);
      }
      }
      
            \foreach \x  in {a,b,c}  {
 \foreach \y in {a,b,c}{
 \foreach \z in {a,b,c} \draw
      (barycentric cs:\x=11,\y=5,\z=2) --
      (barycentric cs:\x=1,\y=1,\z=0);
      }
      }
      
 \foreach \x  in {a,b,c}  {
 \foreach \y in {a,b,c}{
 \foreach \z in {a,b,c}
       \draw (barycentric cs:\x=11,\y=5,\z=2) --  (barycentric cs:\x=5,\y=5,\z=2);
      }
      }
      
\filldraw[blue] (barycentric cs:a=1,b=0,c=0) circle (3pt);
\filldraw[blue] (barycentric cs:a=0,b=1,c=0) circle (3pt);
\filldraw[red](barycentric cs:a=0,b=0,c=1) circle (3pt);

\filldraw[red](barycentric cs:a=1,b=1,c=1) circle (3pt);

\filldraw[red](barycentric cs:a=0,b=1,c=1) circle (3pt);
\filldraw[red](barycentric cs:a=0,b=1,c=3) circle (3pt);
\filldraw[red](barycentric cs:a=0,b=3,c=1) circle (3pt);

\filldraw[red](barycentric cs:a=1,b=0,c=1) circle (3pt);
\filldraw[red](barycentric cs:a=3,b=0,c=1) circle (3pt);
\filldraw[red](barycentric cs:a=1,b=0,c=3) circle (3pt); 

\filldraw[blue](barycentric cs:a=1,b=1,c=0) circle (3pt);
\filldraw[blue](barycentric cs:a=3,b=1,c=0) circle (3pt);
\filldraw[blue](barycentric cs:a=1,b=3,c=0) circle (3pt); 

 \filldraw[red] (barycentric cs:a=5,b=5,c=2) circle (3pt);
  \filldraw[red] (barycentric cs:a=5,b=2,c=5) circle (3pt);
   \filldraw[red] (barycentric cs:a=2,b=5,c=5) circle (3pt);

 \filldraw[red] (barycentric cs:a=5,b=11,c=2) circle (3pt);
  \filldraw[red] (barycentric cs:a=5,b=2,c=11) circle (3pt);
   \filldraw[red] (barycentric cs:a=2,b=5,c=11) circle (3pt);
    \filldraw[red] (barycentric cs:a=2,b=11,c=5) circle (3pt);
  \filldraw[red] (barycentric cs:a=11,b=2,c=5) circle (3pt);
   \filldraw[red] (barycentric cs:a=11,b=5,c=2) circle (3pt);
   
    \filldraw[red] (barycentric cs:a=4,b=1,c=1) circle (3pt);
  \filldraw[red] (barycentric cs:a=1,b=4,c=1) circle (3pt);
   \filldraw[red] (barycentric cs:a=1,b=1,c=4) circle (3pt);
\end{tikzpicture}
\end{subfigure}
\begin{subfigure}[b]{0.45\textwidth}
\centering
\begin{tikzpicture}[scale=0.9]
\path
      (90:3cm) coordinate (a) 
      (210:3cm) coordinate (b) 
      (-30:3cm) coordinate (c);  
      
\draw (a)--(b)--(c) --cycle;
      
\foreach \x  in {a,b,c}  {
\foreach \y in {a,b,c}{
\foreach \z in {a,b,c} \draw
      (barycentric cs:\x=1,\y=1,\z=1) --
      (barycentric cs:\x=1,\y=0,\z=0);
      }
      }
      
\foreach \x  in {a,b,c}  {
\foreach \y in {a,b,c}{
\foreach \z in {a,b,c} \draw
      (barycentric cs:\x=1,\y=1,\z=1) --
      (barycentric cs:\x=1,\y=1,\z=0);
      }
      }
      
\foreach \x  in {a,b,c}{
\foreach \y in {a,b,c}{
\foreach \z in {a,b,c} \draw
      (barycentric cs:\x=11,\y=5,\z=2) --
      (barycentric cs:\x=1,\y=0,\z=0);
      }
      }
      
            \foreach \x  in {a,b,c}  {
 \foreach \y in {a,b,c}{
 \foreach \z in {a,b,c} \draw
      (barycentric cs:\x=11,\y=5,\z=2) --
      (barycentric cs:\x=4,\y=1,\z=1);
      }
      }

        \foreach \x  in {a,b,c}  {
 \foreach \y in {a,b,c}{
 \foreach \z in {a,b,c} \draw
      (barycentric cs:\x=11,\y=5,\z=2) --
      (barycentric cs:\x=1,\y=1,\z=1);
      }
      }
      
          \foreach \x  in {a,b,c}  {
 \foreach \y in {a,b,c}{
 \foreach \z in {a,b,c} \draw
      (barycentric cs:\x=11,\y=5,\z=2) --
      (barycentric cs:\x=3,\y=1,\z=0);
      }
      }
      
            \foreach \x  in {a,b,c}  {
 \foreach \y in {a,b,c}{
 \foreach \z in {a,b,c} \draw
      (barycentric cs:\x=11,\y=5,\z=2) --
      (barycentric cs:\x=1,\y=1,\z=0);
      }
      }
      
 \foreach \x  in {a,b,c}  {
 \foreach \y in {a,b,c}{
 \foreach \z in {a,b,c}
       \draw (barycentric cs:\x=11,\y=5,\z=2) --  (barycentric cs:\x=5,\y=5,\z=2);
      }
      }
      
\filldraw[blue] (barycentric cs:a=1,b=0,c=0) circle (3pt);
\filldraw[blue] (barycentric cs:a=0,b=1,c=0) circle (3pt);
\filldraw[red](barycentric cs:a=0,b=0,c=1) circle (3pt);

\filldraw[red](barycentric cs:a=1,b=1,c=1) circle (3pt);

\filldraw[red](barycentric cs:a=0,b=1,c=1) circle (3pt);
\filldraw[red](barycentric cs:a=0,b=1,c=3) circle (3pt);
\filldraw[red](barycentric cs:a=0,b=3,c=1) circle (3pt);

\filldraw[red](barycentric cs:a=1,b=0,c=1) circle (3pt);
\filldraw[red](barycentric cs:a=3,b=0,c=1) circle (3pt);
\filldraw[red](barycentric cs:a=1,b=0,c=3) circle (3pt); 

\filldraw[blue](barycentric cs:a=1,b=1,c=0) circle (3pt);
\filldraw[blue](barycentric cs:a=3,b=1,c=0) circle (3pt);
\filldraw[blue](barycentric cs:a=1,b=3,c=0) circle (3pt); 

 \filldraw[red] (barycentric cs:a=5,b=5,c=2) circle (3pt);
  \filldraw[red] (barycentric cs:a=5,b=2,c=5) circle (3pt);
  
   \filldraw[blue] (barycentric cs:a=2,b=5,c=5) circle (3pt);

\filldraw[red] (barycentric cs:a=5,b=2,c=11) circle (3pt);
\filldraw[red] (barycentric cs:a=11,b=2,c=5) circle (3pt);
\filldraw[red] (barycentric cs:a=11,b=5,c=2) circle (3pt);
\filldraw[red] (barycentric cs:a=2,b=11,c=5) circle (3pt);
\filldraw[red] (barycentric cs:a=5,b=11,c=2) circle (3pt);

\filldraw[green] (barycentric cs:a=2,b=5,c=11) circle (3pt);
   
\filldraw[red] (barycentric cs:a=4,b=1,c=1) circle (3pt);
\filldraw[red] (barycentric cs:a=1,b=4,c=1) circle (3pt);
\filldraw[red] (barycentric cs:a=1,b=1,c=4) circle (3pt);

\end{tikzpicture}
\end{subfigure}
\caption{\small{The values of a simplicial map $h^*$ on $\sd^2(\sigma_k)_{\bullet}$ (left) and the corresponding values of $h$ (right). The red vertices are sent in $\omega_{0,k}$, the blue ones in $\omega_{1,k}$ and the green ones in $\omega_{2,k}$, where $\tau_k=\Conv(\{\omega_{0,k},\omega_{1,k},\omega_{2,k}\})$ and $\omega_{0,k}<_{\Ll}\omega_{1,k}<_{\Ll}\omega_{2,k}$.}}
\label{accastar3}
\end{figure}
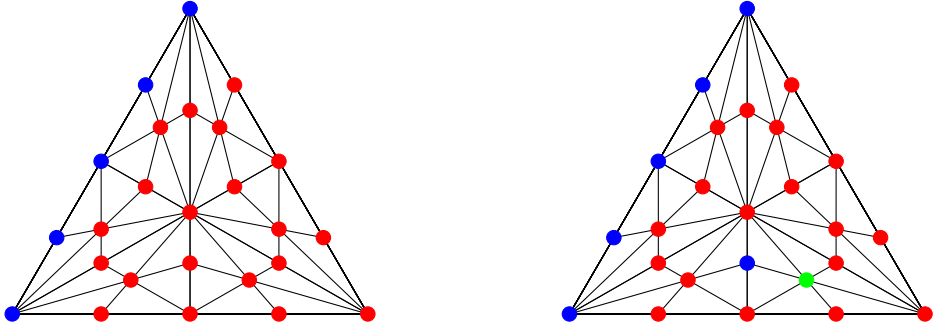

First, we show (i). Let $k\in\{1,\ldots,s\}$. As $h$ is piecewise affine and $r_k\geq d_k$, we have 
\begin{multline*}
h(\sigma'_k)=h(\Conv(\{\nu_{0,k},\ldots,\nu_{r_k,k}\}))=\Conv(\{h(\nu_{0,k}),\ldots,h(\nu_{r_k,k})\})\\
=\Conv(\{\omega_{0,k},\ldots,\omega_{d_k,k}\})=\tau_k,
\end{multline*}
so $\tau_k=h(\sigma'_k)\subset h(\sigma_k)$. Let $\sigma\in\pol^*$ be a simplex such that $\sigma\subset \sigma_k$ and let $\nu$ be a vertex of $\sigma$. Then, as $h^*(\sigma)\subset h^*(\sigma_k)$, $h^*:|\pol^*|\to|\Ll|$ is a simplicial map and $h^*(\sigma_k)$ is a face of $\tau_k$ (see {\sc Step 5}), we have that $h^*(\nu)$ is a vertex of $\tau_k$. In particular, by the definition of $h$, it follows that also $h(\nu)$ is a vertex of $\tau_k$. We deduce, $h(\sigma)\subset\tau_k$. As $\pol^*$ is a refinement of $\pol$ and $\sigma_k\in \pol$, by \eqref{bigcup}, we have $h(\sigma_k)\subset\tau_k$. Thus, $h(\sigma_k)=\tau_k$, as desired.

Then, we show (ii). Let $\nu\in\pol^*_{\bullet}\setminus(\sigma_1\cup\ldots\cup\sigma_s)$. As $\st(\nu,\pol^*)\cap \sigma=\varnothing$ for each simplex $\sigma\in \pol^*$ such that $\nu$ is not a vertex of $\sigma$, by \eqref{bigcup}, we deduce 
$$
\st(\nu,\pol^*)\cap\sigma_k=\st(\nu,\pol^*)\cap\bigcup_{\{\sigma\in\pol^* : \sigma\subset\sigma_k\}}\sigma=\bigcup_{\{\sigma\in\pol^* : \sigma\subset\sigma_k\}}(\st(\nu,\pol^*)\cap\sigma)=\varnothing.
$$
for each $k=1,\ldots,s$. Thus, $\st(\nu,\pol^*)\subset |\pol|\setminus(\sigma_1\cup\ldots\cup\sigma_s)$. As $\sigma_1,\ldots,\sigma_s$ are maximal simplices of $\pol$, their (relative) interiors $\ov{\sigma_1},\ldots,\ov{\sigma_s}$ are open in $|\pol|$, so
\begin{align}\label{closedstar}
\begin{split}
\ol{\st}(\nu,\pol^*)&=\cl(\st(\nu,\pol^*))\subset\cl(|\pol|\setminus(\sigma_1\cup\ldots\cup\sigma_s))\\
&=|\pol|\setminus(\ov{\sigma_1}\cup\ldots\cup\ov{\sigma_s})\subset|\pol|\setminus(\sigma'_1\cup\ldots\cup\sigma'_s)=|\pol^*|\setminus(\sigma'_1\cup\ldots\cup\sigma'_s).
\end{split}
\end{align}
Let $x\in \ol{\st}(\nu,\pol^*)$ and let $\sigma\in\pol^*$ be a simplex such that $x\in\sigma$. Then, 
$$
x=\lambda_0\nu_0+\ldots+\lambda_r\nu_r,
$$
where $\nu_0,\ldots,\nu_r$ are the vertices of $\sigma$ and $(\lambda_0,\ldots,\lambda_r)$ are the barycentric coordinates of $x$. As $\sigma\subset \ol{\st}(\nu,\pol^*)$ and $h(\nu)=h^*(\nu)$ for each $\nu\in \pol^*_{\bullet}\setminus (\sigma'_1\cup\ldots\cup\sigma'_s)$, by \eqref{closedstar}, we have $h(\nu_i)=h^*(\nu_i)$ for each $i=0,\ldots,r$. In particular, as both $h$ and $h^*$ are piecewise affine, we deduce
\begin{align*}
h(x)&=h(\lambda_0\nu_0+\ldots+\lambda_r\nu_r)=\lambda_0 h(\nu_0)+\ldots+\lambda_r h(\nu_r)\\
&=\lambda_0 h^*(\nu_0)+\ldots+\lambda_r h^*(\nu_r)=h^*(\lambda_0\nu_0+\ldots+\lambda_r\nu_r)=h^*(x).
\end{align*}
Thus, $h=h^*$ on $\ol{\st}(\nu,\pol^*)$ for each $\nu\in\pol^*_{\bullet}\setminus(\sigma_1\cup\ldots\cup\sigma_s)$. To conclude, observe that as $\sigma_1,\ldots,\sigma_s$ are maximal simplices of $\pol$ and $\pol^*$ is a refinement of $\pol$, then $\pol^*\setminus(\ov{\sigma}_1\cup\ldots\cup\ov{\sigma}_s)$ is a subcomplex of $\pol^*$ and
$$
|\pol^*|\setminus(\ov{\sigma}_1\cup\ldots\cup\ov{\sigma}_s)=\bigcup_{\nu\in \pol^*_{\bullet}\setminus(\sigma_1\cup\ldots\cup\sigma_s)}\ol{\st}(\nu,\pol^*).
$$
Thus, $h=h^*$ on $|\pol^*|\setminus(\ov{\sigma}_1\cup\ldots\cup\ov{\sigma}_s)=|\pol|\setminus(\ov{\sigma}_1\cup\ldots\cup\ov{\sigma}_s)$, as required.

It only remains to show: \textit{$h:|\pol^*|=|\pol|\to|\Ll|$ is a simplicial map.} By (i) and (ii), we have
\begin{align*}
h(|\pol|)&=h((|(\pol|\setminus(\ov{\sigma}_1\cup\ldots\cup\ov{\sigma}_s))\cup (\sigma_1\cup\ldots\cup\sigma_s))\\
&= h(|\pol|\setminus(\ov{\sigma}_1\cup\ldots\cup\ov{\sigma}_s))\cup \bigcup_{k=1}^s h(\sigma_k)=h^*(|\pol|\setminus(\ov{\sigma}\cup\ldots\cup\ov{\sigma}_s))\cup\bigcup_{k=1}^s \tau_k\subset|\Ll|,
\end{align*}
so $h(|\pol^*|)=h(|\pol|)\subset|\Ll|$. Let now $\sigma$ be a simplex of $\pol^*$. If $\sigma\cap(\ov{\sigma}_1\cup\ldots\cup\ov{\sigma}_s)=\varnothing$, then $h(\sigma)\in\Ll$, because $h(\sigma)=h^*(\sigma)$ and $h^*(\sigma)\in\Ll$, because $h^*:|\pol^*|\to|\Ll|$ is a simplicial map. If $\sigma\cap\ov{\sigma}_k\neq\varnothing$ for some $k\in\{1,\ldots,s\}$, then $\sigma\subset\sigma_k$ because $\sigma_k$ is a maximal simplex of $\pol$ and $\pol^*$ is a refinement of $\pol$. Thus, by the definition of $h$, we have that $h(\nu)$ is a vertex of $\tau_k$ for each vertex $\nu$ of $\sigma$. In particular, if $\nu_0,\ldots,\nu_r$ are the vertices of $\sigma$, then
$$
h(\sigma)=h(\Conv(\{\nu_0,\ldots,\nu_r\}))=\Conv(\{h(\nu_0),\ldots,h(\nu_r)\})\subset\tau_k
$$
is a face of $\tau_k$, so $h(\sigma)\in\Ll$. We conclude that $h:|\pol^*|=|\pol|\to|\Ll|$ is a simplicial map, as required.

\noindent{\sc Step 7.} In order to conclude, it only remains to show: \textit{$h$ is surjective and satisfies 
\begin{equation}\label{star2}
f(\st(\nu,\pol^*))\subset\st^{(2)}(h(\nu),\Ll)
\end{equation}
for each vertex $\nu\in\pol^*_{\bullet}$.}  

By (i), we have $\tau_1,\ldots,\tau_s\subset h(|\pol^*|)$. Thus, 
$$
|\Ll|=\tau_1\cup\ldots\cup\tau_s\subset h(|\pol^*|)\subset|\Ll|,
$$
so $h$ is surjective.

Next, we show \eqref{star2}. Let $\nu\in \pol^*_{\bullet}$ be a vertex. If $\nu\not\in |\pol^*_{\bullet}|\cap(\sigma'_1\cup\ldots\cup\sigma'_s)$, then by the definition of $h$, we have $h(\nu)=h^*(\nu)$. Thus,
$$
f(\st(\nu,\pol^*))\subset \st(h^*(\nu),\Ll)=\st(h(\nu),\Ll)\subset\st^{(2)}(h(\nu),\Ll),
$$
because $h^*$ is a simplicial approximation of $f$. Assume now that $\nu\in \sigma'_k$ for some $k=1,\ldots,s$. Let $\omega_{0,k},\ldots,\omega_{d_k,k}$ be the vertices of $\tau_k$ and $j\in\{0,\ldots,d_k\}$. As $\tau_k\subset \ol{\st}(\omega_{j,k},\Ll)$, it follows $\{\omega_{0,k},\ldots,\omega_{d_k,k}\}\subset |\Ll_{\bullet}|\cap\ol{\st}(\omega_{j,k},\Ll)$, so
$$
\bigcup_{i=0}^{d_k}\st(\omega_{i,k},\Ll)\subset \bigcup_{\omega\in|\Ll_{\bullet}|\cap\ol{\st}(\omega_{j,k},\Ll)}\st(\omega,\Ll)=\st^{(2)}(\omega_{j,k},\Ll).
$$
Moreover, as $h^*:|\pol^*|\to|\Ll|$ is a simplicial map, $\nu$ is a vertex of $\sigma_k'\in\pol^*$ and $h^*(\sigma_k)\subset\tau_k$, we deduce that $h^*(\nu)$ is a vertex of $\tau_k$. In particular, $h^*(\nu)\in \{\omega_{0,k},\ldots,\omega_{d_k,k}\}$. Thus, as $h^*$ is a simplicial approximation of $f$, we have
$$
f(\st(\nu,\pol^*))\subset \st(h^*(\nu),\Ll)\subset\bigcup_{i=0}^{d_k}\st(\omega_{i,k},\Ll)\subset\st^{(2)}(\omega_{j,k},\Ll).
$$
As $h(\nu)$ is a vertex of $\tau_k$, then $h(\nu)=\omega_{\ell,k}$ for some $\ell\in\{0,\ldots,d_k\}$. We conclude
$$
f(\st(\nu,\pol^*))\subset\st^{(2)}(\omega_{\ell,k},\Ll)=\st^{(2)}(h(\nu),\Ll),
$$
as required.
\end{proof}

We are ready to show Proposition \ref{key3}.

\begin{proof}[Proof of Proposition \ref{key3}]
Let $\ell\geq 0$ be an integer such that $\diam(\tau)<\tfrac{\veps}{3}$ for each simplex $\tau\in\sd^{\ell}(\Ll)$ \cite[Thm.15.4]{mu}. By Proposition \ref{simpapprox}, there exists an integer $\kappa\geq 0$ and a surjective simplicial map $h:|\sd^{\kappa}(\pol)|\to|\sd^{\ell}(\Ll)|$ such that
\begin{equation}\label{star22}
f(\st(\nu,\sd^{\kappa}(\pol)))\subset\st^{(2)}(h(\nu),\sd^{\ell}(\Ll))
\end{equation}
for each vertex $\nu\in\sd^{\kappa}(\pol)_{\bullet}$. In order to lighten the notation we replace $\pol$ with $\sd^k(\pol)$ and $\Ll$ with $\sd^\ell(\Ll)$. Let $x\in|\pol|$. As
$$
|\pol|=\bigcup_{\nu\in\pol_{\bullet}}\st(\nu,\pol),
$$
there exists a vertex $\nu\in\pol_{\bullet}$ such that $x\in\st(\nu,\pol)$. We have,
$$
\dist(f(x),\ol{\st}(h(\nu),\Ll))\leq\max_{\tau\in\Ll}\{\diam(\tau)\})<\frac{\veps}{3},
$$
because $f(x)\in\st^{(2)}(h(\nu),\Ll)$. In particular, 
\begin{align*}
|f(x)-h(\nu)|_m&\leq \max_{z\in \ol{\st}(h(\nu),\Ll)}|h(\nu)-z|_m+\dist(f(x),\ol{\st}(h(\nu),\Ll))\\
&<\max_{\tau\in\Ll}\{\diam(\tau)\}+\frac{\veps}{3}=\frac{2\veps}{3}.
\end{align*}
As $h$ is a simplicial map, then $h(x)\in \st(h(\nu),\Ll)$. Thus, 
\begin{align*}
|f(x)-h(x)|_m&\leq|f(x)-h(\nu)|_m+|h(\nu)-h(x)|_m\\
&<\frac{2\veps}{3}+\max_{\tau\in\Ll}\{\diam(\tau)\}<\frac{2\veps}{3}+\frac{\veps}{3}=\veps.
\end{align*}
We conclude $\|f-h\|<\veps$, as desired.
\end{proof}

\section{Surjective definable approximation between polyhedra}

The purpose of this section is to show the following proposition. Together with Proposition \ref{key3} it will be the key ingredient for the proof of Theorem \ref{main}. The $\eps$-squeezing map of a finite simplicial complex is defined in \S\ref{squeezing} below. Roughly speaking the $\eps$\textit{-squeezing map} of a finite simplicial complex $\Ll$ is a surjective continuous definable map $\pi_{\eps}:|\Ll|\to|\Ll|$ that `squeezes' a small neighbourhood of the boundary of each maximal simplex of $\Ll$ onto its boundary and that is the identity map on each simplex of $\Ll$ that is not maximal.

\begin{prop}\label{key}
Let $\pol$ be a finite simplicial complex of $\R^n$, $\Ll$ a finite simplicial complex of $\R^m$ and $h:|\pol|\to|\Ll|$ a surjective simplicial map. Then, there exists $\eps>0$ with the following property: if $g:|\pol|\to |\Ll|$ is a continuous map such that $\|h-g\|<\eps$, then
\begin{itemize}
\item $\pi_{\eps}\circ g:|\pol|\to |\Ll|$ is surjective, where $\pi_{\eps}$ is the $\eps$-squeezing map of $\Ll$,
\item $\|h-\pi_{\eps}\circ g\|<2\displaystyle\max_{\tau\in\Ll}\{\diam(\tau)\}$.
\end{itemize}
\end{prop}

\subsection{$\eps$-squeezing map}\label{squeezing}

We start by recalling the following well-known result, that we include here for the sake of completeness. The proof is taken from \cite[Lem.1.1]{fu} and adapted to our situation.

\begin{lem}\label{easyretraction}
Let $d\geq1$ be an integer and $\tau\subset \R^d$ a simplex of dimension $d$. Then, for each $z\in \ov{\tau}$ there exists a continuous definable retraction $\rho_z:\R^d\setminus\{z\}\to \partial \tau$.
\end{lem}
\begin{proof}
Let $z\in \ov{\tau}$. For each $x\in\R^d\setminus\{z\}$ let $\ell_x:=\{z+t(x-z) : t\in[0,+\infty)\}$ be the ray from $z$ that passes through $x$. As $\tau\subset\R^d$ has dimension $d$, then $\ell_x\cap \partial\tau\neq\varnothing$. Moreover, as $\tau$ is convex, by \cite[11.1.2.3, 11.1.2.7]{be}, $\ell_x\cap \partial\tau$ is a single point. Define $\rho: \R^d\setminus\{z\}\to \partial \tau$ as $\rho(x):=\ell_x\cap\partial\tau$. Observe that $\rho(x)=x$ for each $x\in\partial\tau$. Let $h_0,\ldots,h_d\in\R[\x_1,\ldots,\x_d]$ be polynomials of degree 1 such that the hyperplanes $H_i:=\{h_i=0\}$ contain the facets of $\tau$. Assume $\tau\subset\{h_i\geq 0\}$ for each $i=0,\ldots,d$. Note that $\rho(x)=z+\lambda(x-z)$, where $\lambda$ is the smallest value $\mu>0$ such that $h_i(z+\mu(x-z))=0$ for some $i=0,\ldots,d$. As $z\in\ov{\tau}$, we have $h_i(z)>0$ for $i=0,\ldots,d$. Thus, 
$$
\frac{1}{\lambda}=\max\Big\{\frac{h_i(z)-h_i(x)}{h_i(z)}:\ i=0,\ldots,d\Big\}>0.
$$
Consequently,
$$
\rho(x)=z+\frac{1}{\max\Big\{\frac{h_i(z)-h_i(x)}{h_i(z)}:\ i=0,\ldots,d\Big\}}(x-z),
$$
so $\rho:\R^n\setminus\{z\}\to\partial\tau$ is a continuous semialgebraic map such that $\rho|_{\partial\tau}=\id_{\partial\tau}$, that is, $\rho$ is a semialgebraic retraction. In particular, $\rho$ is a retraction that is definable in every o-minimal structure expanding the ordered field of real numbers $\R$.
\end{proof}

Let $\Ll$ be a finite simplicial complex of $\R^m$ and $\tau\in \Ll$ a simplex of dimension $d\geq 1$. Let $b_{\tau}$ be the barycentre of $\tau$ and $\delta_\tau:=\dist(b_{\tau},\partial\tau)$. Observe that, as $d\geq 1$, then $\delta_{\tau}>0$. Let $H_{\tau}$ be the affine space generated by $\tau$, that is, the smallest affine subspace $H_{\tau}$ of $ \R^m$ that contains $\tau$. For each $\eps\in(0,\delta_{\tau})$ denote 
\begin{equation}\label{homothety}
\theta_{\tau,\eps}: H_{\tau}\to H_{\tau}, \, x\mapsto b_{\tau}+\frac{\eps}{\delta_{\tau}}(x-b_{\tau})
\end{equation}
the homothety of $H_{\tau}$ of centre $b_{\tau}$ and radius $\eps/\delta_{\tau}$. Observe that $\theta_{\tau,\eps}(\tau)\subset\tau$, because $\eps/\delta_\tau<1$. As $\dim(\tau)=\dim(H_\tau)$, by Lemma \ref{easyretraction}, there exists a continuous definable retraction $\rho_{b_{\tau}}:H_{\tau}\setminus\{b_{\tau}\}\to \partial \tau$. Define 
\begin{equation}\label{pitau}
\pi_{\tau,\eps}:\tau\to\tau,\, x\mapsto
\begin{cases}
\theta^{-1}_{\tau,\eps}(x)=\theta_{\tau,\tfrac{1}{\eps}}(x), \text{ if } x\in \theta_{\tau,\eps}(\tau), \\
\rho_{b_\tau}(x), \text{ if } x\in\tau\setminus \theta_{\tau,\eps}(\ov{\tau}).
\end{cases}
\end{equation}
It is clear that $\pi_{\tau,\eps}$ is a continuous definable map such that $\pi_{\tau,\eps}(\tau)=\tau$ and the restriction $\pi_{\tau,\eps}|_{\partial\tau}$ to $\partial \tau$ is the identity map $\id_{\partial\tau}$ of $\partial\tau$.

Let $\Lambda_{\Ll}\subset\Ll$ be the family of maximal simplices of $\Ll$ of dimension $\geq 1$. Define
$$
\delta_{\Ll}:=\min_{\tau\in\Lambda_{\Ll}}\{\delta_{\tau}\}=\min_{\tau\in\Lambda_{\Ll}}\{\dist(b_{\tau},\partial\tau)\}.
$$
Then, for each $\tau\in\Lambda_{\Ll}$ and $\eps\in(0,\delta_{\Ll})$, the map $\pi_{\tau,\eps}:\tau\to \tau$ extends to a continuous definable map $\widehat{\pi}_{\tau,\eps}:|\Ll|\to |\Ll|$ that is the identity on $|\Ll|\setminus\ov{\tau}$. In fact, as $\tau$ is a maximal simplex of $\Ll$, then $\ov{\tau}$ is open in $|\Ll|$, so $|\Ll|\setminus \ov{\tau}$ is closed. As $\tau\cap(|\Ll|\setminus\ov{\tau})=\partial\tau$ and $\pi_{\tau,\eps}|_{\partial\tau}=\id_{\partial\tau}$, the map 
$$
\widehat{\pi}_{\tau,\eps}:|\Ll|\to|\Ll|,\, x\mapsto 
\begin{cases}
\pi_{\tau,\eps}(x), \text{ if } x\in \tau \\
x, \text{ if } x\in|\Ll|\setminus\ov{\tau},
\end{cases}
$$
is a well-defined continuous definable map. 

Let $\Lambda_{\Ll}=\{\tau_1,\ldots,\tau_s\}$ and $\eps\in(0,\delta_{\Ll})$. We are ready to give the definition of $\eps$-squeezing map of a finite simplicial complex $\Ll$.

\begin{defn}\label{defsqueezing}
The \textit{$\eps$-squeezing map of $\Ll$} is the continuous definable map 
$$
\pi_{\eps}:=\widehat{\pi}_{\tau_s,\eps}\circ\cdots\circ\widehat{\pi}_{\tau_1,\eps}:|\Ll|\to|\Ll|. 
$$
\end{defn}

Observe that the map $\pi_{\eps}$ depends on the simplicial complex $\Ll$ and not only on the polyhedron $|\Ll|$. Moreover, as $\widehat{\pi}_{\tau_i,\eps}\circ\widehat{\pi}_{\tau_j,\eps}=\widehat{\pi}_{\tau_j,\eps}\circ\widehat{\pi}_{\tau_i,\eps}$ for each $\tau_i,\tau_j\in\Lambda_{\Ll}$, the map $\pi_{\eps}$ does not depend on how we order the set $\Lambda_{\Ll}$, so it is completely determined by $\eps\in(0,\delta_{\Ll})$ and the simplicial complex $\Ll$.

We end this section with the following remark.

\begin{remark}\label{partialpi}
\textit{We have $\pi_\eps(\tau)=\tau$ for each $\tau\in\Ll$ and $\eps\in(0,\delta_{\Ll})$.} In fact, let $\tau\in\Ll$ and $\tau_j\in\Lambda_{\Ll}$. As $\tau_j$ is a maximal simplex of $\Ll$, then either $\tau=\tau_j$ or $\tau\subset |\Ll|\setminus\ov{\tau}_j$. In particular, $\widehat{\pi}_{\tau_j,\eps}(\tau)=\tau$ for each $\eps\in(0,\delta_{\Ll})$, because $\widehat{\pi}_{\tau_j,\eps}(\tau_j)=\tau_j$ and $\widehat{\pi}_{\tau_j,\eps}$ is the identity map on $|\Ll|\setminus\ov{\tau}_j$. As $\pi_{\eps}$ is the composition of the the maps $\widehat{\pi}_{\tau_1,\eps},\ldots,\widehat{\pi}_{\tau_s,\eps}$, we conclude that $\pi_{\eps}(\tau)=\tau$ for each $\eps\in(0,\delta_{\Ll})$. $\sqbullet$
\end{remark}

\subsection{Proof of Proposition \ref{key}} In this section we show Proposition \ref{key}. Let $\pi_{\tau,\eps}$ be the map introduced in \eqref{pitau}, we start with the following lemma.

\begin{lem}\label{prepkey}
Let $\tau\subset \R^d$ be a simplex of dimension $d\geq 1$ and $\delta_{\tau}:=\dist(b_{\tau},\partial\tau)>0$. Then, there exists $\eps^*\in(0,\delta_{\tau})$ with the following property: if $g:\tau\to\tau$ is a continuous map such that 
$$
\|\id_{\partial\tau}-g|_{\partial\tau}\|<\frac{1}{2}\delta_{\tau},
$$
then $\pi_{\tau,\eps}(g(\tau))=\tau$ for each $\eps\in(0,\eps^*)$. 
\end{lem}
\begin{proof}
Let $\eps\in(0,\delta_{\tau})$ and let $\theta_{\eps}:\R^d\to\R^d$ be the homothety of centre $b_{\tau}$ and radius $\eps/\delta_{\tau}$ (see \eqref{homothety}). As $\pi_{\tau,\eps}(\theta_{\eps}(\tau))=\tau$, it is enough to show that there exists $\eps^*\in(0,\delta_{\tau})$ such that $\theta_{\eps}(\tau)\subset g(\tau)$ for each $0<\eps<\eps^*$. Suppose, by the way of contradiction, that for all $\eps\in (0,\delta_{\tau})$ there exists a point $z_{\eps}\in \theta_{\eps}(\tau)\setminus g(\tau)$. As $z_{\eps}\in \theta_{\eps}(\tau)\subset\ov{\tau}$, by Lemma \ref{easyretraction}, there exists a (continuous) retraction $\rho_{z_{\eps}}:\R^d\setminus\{z_{\eps}\}\to \partial\tau$. We claim: \textit{There exists $\eps^*\in(0,\delta_{\tau})$ such that the (continuous) map $\widehat{g}_{\eps}:=\rho_{z_{\eps}}\circ g|_{\partial\tau}:\partial\tau\to\partial\tau$ is homotopic to the identity map $\id_{\partial\tau}$ of $\partial\tau$ for all $0<\eps<\eps^*$.} 

For each $x\in \partial\tau$ let
$$
\s_{x,\eps}:=\{t\widehat{g}_{\eps}(x)+(1-t)x : 0\leq t\leq 1\}
$$
be the segment between $x$ and $\widehat{g}_{\eps}(x)$. Let $\rho_{b_{\tau}}:\R^d\setminus\{b_{\tau}\}\to \partial \tau$ be the (continuous) retraction relative to the barycentre $b_{\tau}$ of $\tau$ (see Lemma \ref{easyretraction}). If $b_{\tau}\not\in\s_{x,\eps}$ for each $x\in \partial\tau$, then the map 
$$
\partial\tau\times[0,1]\to\partial\tau, \, (x,t)\mapsto \rho_{b_{\tau}}(t\widehat{g}_{\eps}(x)+(1-t)x)
$$
is a well-defined homotopy between the identity map $\id_{\partial\tau}$ and $\widehat{g}_{\eps}$. If there exists $x\in \partial\tau$ such that $b_{\tau}\in \s_{x,\eps}$, then
\begin{align*}
\delta_{\tau}+\dist(b_\tau,\widehat{g}_\eps(x)) &= \min_{y\in\partial\tau}\{\dist(y,b_\tau)\}+\dist(b_\tau,\widehat{g}_\eps(x))\\
&\leq \dist(x,b_{\tau})+\dist(b_\tau,\widehat{g}_\eps(x))=\dist(x,\widehat{g}_{\eps}(x))\\
&\leq \dist(x,g(x))+\dist(g(x),\widehat{g}_{\eps}(x))< \frac{\delta_{\tau}}{2}+\dist(z_{\eps},\widehat{g}_{\eps}(x))\\
&\leq \frac{\delta_{\tau}}{2}+\dist(z_{\eps},b_{\tau})+\dist(b_{\tau},\widehat{g}_{\eps}(x))\\
&\leq \frac{\delta_{\tau}}{2}+\frac{\eps}{\delta_{\tau}}\diam(\tau)+\dist(b_\tau,\widehat{g}_\eps(x))
\end{align*}
The inequality $\dist(g(x),\widehat{g}_{\eps}(x))\leq \dist(z_{\eps},\widehat{g}_{\eps}(x))$ holds because $g(x)$ belongs to the segment between $z_{\eps}$ and $\widehat{g}_{\eps}(x)$. We deduce
\begin{equation}\label{ineq}
\delta_\tau<\frac{\delta_\tau}{2}+\frac{\eps}{\delta_{\tau}}\diam(\tau)
\end{equation}
Define the number
$$
\eps^*:=\frac{\delta_{\tau}^2}{2\diam(\tau)}>0
$$
and observe that, as $\delta_\tau<\diam(\tau)$, then $\eps^*<\delta_\tau$. If $\eps<\eps^*$, the inequality \eqref{ineq} becomes  
$$
\delta_{\tau}\leq\frac{\delta_\tau}{2}+\frac{\eps}{\delta_{\tau}}\diam(\tau)<\delta_{\tau},
$$
which leads to a contradiction. Thus, if $\eps<\eps^*$, then $b_{\tau}\not\in \s_{x,\eps}$ for each $x\in \partial \tau$, so $\widehat{g}_{\eps}$ is homotopic to the identity map $\id_{\partial\tau}$, as claimed.

The map $\widehat{g}_{\eps}$ is the restriction to $\partial\tau$ of the continuous map $\rho_{z_{\eps}}\circ g:\tau\to\partial\tau$. Thus, the map induced by $\widehat{g}_{\eps}$ on homology
$$
(\widehat{g}_{\eps})_*:H_d(\partial\tau,\Z)\simeq\Z\to H_d(\partial\tau,\Z)\simeq\Z
$$
is zero, because it factors through $H_d(\tau,\Z)=0$. If $\eps<\eps^*$ this leads to a contradiction. In fact,  $(\widehat{g}_{\eps})_*=\id_{H_d(\partial\tau,\Z)}$ because $\widehat{g}_{\eps}$ is homotopic to the identity map $\id_{\partial\tau}$. In particular, if $\eps<\eps^*$, then $\theta_{\eps}(\tau)\subset g(\tau)$. We deduce $\tau=\pi_{\tau,\eps}(\theta_{\eps}(\tau))\subset\pi_{\tau,\eps}(g(\tau))\subset\tau$, so $\pi_{\tau,\eps}(g(\tau))=\tau$, as required.
\end{proof}

We are ready to show Proposition \ref{key}.

\begin{proof}[Proof of Proposition \ref{key}]
The proof is conducted in three steps.

\noindent{\sc Step 1. Initial reduction.} Let $g:|\pol|\to|\Ll|$ be a continuous map such that $\|h-g\|<\eps$. As $\pi_{\eps}(\tau)=\tau$ for each $\tau\in \Ll$ (see Remark \ref{partialpi}), we have
$$
\|h-\pi_{\eps}\circ g\|\leq \|h-g\|+\|g-\pi_{\eps}\circ g\|< \eps+\max_{\tau\in\Ll}\{\diam(\tau)\}
$$
In particular, if $\eps<\max\{\diam(\tau) : \tau\in\Ll\}$, then $\|h-\pi_{\eps}\circ g\|< 2\max\{\diam(\tau): \tau\in\Ll\}$.

As $|\Ll|$ has finitely many connected components, we may assume that $|\Ll|$ is connected. If $\tau\in \Ll_{\bullet}$ is a vertex that is also a maximal simplex of $\Ll$, then $\tau$ is an isolated point of $|\Ll|$, so $|\Ll|=\tau$, because $|\Ll|$ is connected. In particular, we may also assume that all the maximal simplices of $\Ll$ have dimension $\geq 1$.

Let $\tau\in\Ll$ be a simplex that is not a maximal simplex of $\Ll$. Then, there exists a maximal simplex $\tau'\in\Ll$ such that $\tau\subset\tau'$. In particular, if $\tau'\subset\pi_{\eps}(g(|\pol|))$, then $\tau\subset\tau'\subset \pi_{\eps}(g(|\pol|))$. Thus, in order to conclude we are reduced to show: \textit{There exists $\eps>0$ such that if $\|h-g\|<\eps$, then $\tau\subset\pi_{\eps}(g(|\pol|))$ for each maximal simplex $\tau\in \Ll$}.

\noindent{\sc Step 2. Choice of $\eps$.} Let $\Lambda_{\Ll}$ be the family of maximal simplices of $\Ll$ and let $\tau\in\Lambda_{\Ll}$. We may assume that $\tau$ is not the only maximal simplex of $\Ll$, because in the case where $\tau$ is the only maximal simplex of $\Ll$ the proof is similar but easier (in what follows it is enough to take $U_\tau:=\tau$, $\rho_\tau=\id_\tau$ and $\eps_1=1$). By {\sc Step 1}, the polyhedron $|\Ll|$ is connected, so
$$
\partial\st(\tau,\Ll):=\ol{\st}(\tau,\Ll)\setminus\st(\tau,\Ll)\neq\varnothing,
$$
where $\st(\tau,\Ll)$ is the (open) star of $\tau$ (see \eqref{starsimp}). By \cite[Lem.70.1]{mu}, there exists a (continuous) retraction $\rho'_\tau:\st(\tau,\Ll)\to \tau$. As $\tau$ and $\partial\st(\tau,\Ll)$ are compact sets, $\tau\cap\partial\st(\tau,\Ll)=\varnothing$ and $\partial\st(\tau,\Ll)\neq\varnothing$, we have
$$
\eps_1:=\min_{\tau\in\Lambda_{\Ll}}\{\dist(\tau,\partial\st(\tau,\Ll))\}>0.
$$
For each $\tau\in\Lambda_{\Ll}$ the set
$
U_{\tau}:=\{x\in |\Ll| : \dist(x,\tau)<\eps_1\}
$
is an open neighbourhood of $\tau$ in $|\Ll|$ such that $U_{\tau}\subset\st(\tau,\Ll)$. In particular, the restriction $\rho_{\tau}:=\rho'_{\tau}|_{U_{\tau}}:U_{\tau}\to\tau$ of the retraction $\rho'_{\tau}$ to $U_{\tau}$ is well-defined for each $\tau\in\Lambda_{\tau}$.

Let $\tau\in\Lambda_{\Ll}$. As $h$ is a surjective simplicial map, there exists a simplex $\sigma\in\pol$ such that $h(\sigma)=\tau$. As $h|_{\sigma}$ is affine and maps the vertices of $\sigma$ onto the vertices of $\tau$, there exists a face $\sigma_{\tau}$ of $\sigma$ such that $h(\sigma_{\tau})=\tau$ and $\dim(\sigma_\tau)=\dim(\tau)$. In particular, $h|_{\sigma_\tau}:\sigma_{\tau}\to\tau$ is an (affine) homeomorphism. Let $g:|\pol|\to|\Ll|$ be a continuous map. If $\|h-g\|<\eps_1$, then $g(\sigma_{\tau})\subset U_{\tau}$, so the map $\rho_{\tau}\circ g|_{\sigma_{\tau}}:\sigma_{\tau}\to\tau$ is continuous and well-defined. 

For each simplex $\tau\in\Lambda_{\Ll}$, let $b_{\tau}$ be its barycentre. We define
$$
\delta_{\Ll}:=\min_{\tau\in\Lambda_{\Ll}}\{\dist(b_{\tau},\partial\tau)\}.
$$
Observe that, as $\dim(\tau)\geq 1$ for each $\tau\in\Lambda_{\Ll}$, then $\delta_{\Ll}>0$. The map 
$$
(\rho_{\tau})_*:\Cont^0(\tau,U_{\tau})\to \Cont^0(\tau,\tau),\, f\mapsto \rho_{\tau}\circ f
$$
is continuous with respect to the compact-open topology. Thus, there exists $\eps_{\tau}>0$ such that if $f\in \Cont^0(\tau,U_{\tau})$ is a continuous map such that $\|\id_{\tau}-f\|<\eps_{\tau}$, then 
$$
\|\id_{\tau}-\rho_{\tau}\circ f\|=\|\rho_{\tau}\circ \id_{\tau}-\rho_{\tau}\circ f\|<\frac{1}{2}\delta_{\Ll}.
$$
We define
$$
\eps_2:=\min_{\tau\in \Lambda_{\Ll}}\{\eps_{\tau}\}>0
$$

Let $g:|\pol|\to|\Ll|$ be a continuous map such that $\|h-g\|<\min\{\eps_1,\eps_2\}$ and let $\tau\in\Lambda_{\tau}$. As $h|_{\sigma_{\tau}}:\sigma_{\tau}\to\tau$ is a homeomorphism, then
$$
\|\id_{\tau}-g\circ h|_{\sigma_{\tau}}^{-1}\|=\|h|_{\sigma_{\tau}}\circ h|_{\sigma_{\tau}}^{-1}-g\circ h|_{\sigma_{\tau}}^{-1}\|=\|h|_{\sigma_{\tau}}-g|_{\sigma_{\tau}}\|\leq\|h-g\|<\min\{\eps_1,\eps_2\}.
$$
In particular, $g\circ h|_{\sigma_{\tau}}^{-1}:\tau\to U_{\tau}$, so the composition $\rho_{\tau}\circ g\circ h|_{\sigma_{\tau}}^{-1}:\tau\to\tau$ is well-defined. Moreover, as $\min\{\eps_1,\eps_2\}\leq \eps_{\tau}$, we have
\begin{align*}
\|\id_{\partial\tau}-\rho_{\tau}\circ g\circ h|_{\partial\sigma_{\tau}}^{-1}\|&\leq \|\id_{\tau}-\rho_{\tau}\circ g\circ h|_{\sigma_{\tau}}^{-1}\|=\|\rho_{\tau}\circ \id_{\tau}-\rho_{\tau}\circ g\circ h|_{\sigma_{\tau}}^{-1}\|\\
&<\frac{1}{2}\delta_{\Ll}\leq\frac{1}{2} \dist(b_{\tau},\partial\tau).
\end{align*}
By Lemma \ref{prepkey}, there exists $\eps^*_{\tau}\in(0,\delta_\tau)$ such that $
\pi_{\tau,\eps}(\rho_{\tau}(g(h|_{\sigma_{\tau}}^{-1}(\tau))))=\tau
$
for each $0<\eps<\eps^*_{\tau}$ (see \eqref{pitau} for the definition of the map $\pi_{\tau,\eps}$). We define
$$
\eps_3:=\frac{1}{2}\min_{\tau\in\Lambda_{\Ll}}\{\eps^*_{\tau}\}\in(0,\delta_{\Ll}).
$$
Let $0<\eps\leq\eps_3$ and let $\pi_{\eps}:|\Ll|\to|\Ll|$ be the $\eps$-squeezing map of $\Ll$. Thus, 
\begin{equation}\label{surj1}
\pi_{\eps}(\rho_{\tau}(g(\sigma_{\tau})))=\pi_{\eps}(\rho_{\tau}(g(h|_{\sigma_{\tau}}^{-1}(\tau))))=\pi_{\tau,\eps}(\rho_{\tau}(g(h|_{\sigma_{\tau}}^{-1}(\tau))))=\tau,
\end{equation}
because  $\rho_{\tau}(g(h|_{\sigma_{\tau}}^{-1}(\tau)))\subset\tau$, $\pi_{\eps}|_{\tau}=\pi_{\tau,\eps}$ and $\eps\leq \eps_3<\eps^*_{\tau}$. We define
$$
\eps:=\min\{\eps_1,\eps_2,\eps_3\}>0.
$$

\noindent{\sc Step 3. Conclusion.} By {\sc Step 1}, up to taking a smaller $\eps$ if needed, in order to conclude it is enough to show: \textit{$\tau\subset\pi_{\eps}(g(|\pol|))$ for each $\tau\in\Lambda_{\Ll}$.} 

Let $\tau\in\Lambda_{\Ll}$. As $\tau$ is a maximal simplex, $\tau'\cap\ov{\tau}=\varnothing$ for each simplex $\tau'\in\Ll$ such that $\tau'\neq \tau$. Let $x\in U_{\tau}\setminus\ov{\tau}$ and let $\tau'\in\Ll$ be a simplex such that $\tau'\neq\tau$ and $x\in\tau'$. As $U_\tau\setminus\ov{\tau}\subset\st(\tau,\Ll)\setminus\ov{\tau}$, by \cite[Lem.70.1]{mu} and its proof, we have $\rho_\tau(x)\in\tau'$. In particular, $\rho_{\tau}(x)\in\tau'\cap\tau$. As $\tau'\cap\ov{\tau}=\varnothing$, it follows $\tau'\cap\tau\subset\partial\tau$, so $\rho_{\tau}(x)\in\partial\tau$. We deduce, $\rho_{\tau}(U_{\tau}\setminus\ov{\tau})\subset\partial\tau$. As $\rho_{\tau}|_{\tau}=\id_{\tau}$, we have
\begin{align}\label{1psi}
\begin{split}
g(\sigma_{\tau})\cap\ov{\tau}&=\rho_{\tau}(g(\sigma_\tau)\cap\ov{\tau})\subset\rho_{\tau}(g(\sigma_\tau))\cap\rho_{\tau}(\ov{\tau})
= \rho_{\tau}(g(\sigma_\tau))\cap\ov{\tau} \\
&=(\rho_{\tau}(g(\sigma_\tau)\cap\ov{\tau})\cup\rho_{\tau}(g(\sigma_\tau)\cap(U_{\tau}\setminus\ov\tau))\cap\ov{\tau}\\
&\subset(\rho_{\tau}(g(\sigma_\tau)\cap\ov{\tau})\cup\partial\tau)\cap\ov{\tau}=\rho_{\tau}(g(\sigma_\tau)\cap\ov{\tau})\cap\ov{\tau}\\
&=(g(\sigma_\tau)\cap\ov{\tau})\cap\ov{\tau}=g(\sigma_\tau)\cap\ov{\tau}.
\end{split}
\end{align}
In particular, $g(\sigma_{\tau})\cap\ov{\tau}=\rho_{\tau}(g(\sigma_\tau))\cap\ov{\tau}$. As $\pi_{\eps}(\tau')=\tau'$ for each simplex $\tau'\in\Ll$ (see Remark \ref{partialpi}), it follows that $\pi_{\eps}(\partial\tau)=\partial\tau$. As $\|g-h\|<\eps<\eps_3$, by \eqref{surj1}, we have $\pi_{\eps}(\rho_{\tau}(g(\sigma_{\tau})))=\tau$. Thus, by \eqref{1psi}, and by the fact that $\rho_{\tau}(g(\sigma_\tau))\subset\tau$, we deduce,
\begin{align*}
\ov{\tau}&=\tau\cap\ov{\tau}=\pi_{\eps}(\rho_{\tau}(g(\sigma_{\tau})))\cap\ov{\tau}
=(\pi_\eps(\rho_{\tau}(g(\sigma_{\tau}))\cap\ov{\tau})\cup\pi_\eps(\rho_{\tau}(g(\sigma_{\tau}))\cap\partial\tau))\cap\ov{\tau}\\
&\subset (\pi_\eps(\rho_{\tau}(g(\sigma_{\tau}))\cap\ov{\tau})\cup\partial\tau)\cap\ov{\tau}
=\pi_\eps(\rho_{\tau}(g(\sigma_{\tau}))\cap\ov{\tau})\cap\ov{\tau}
\\
&=\pi_\eps(g(\sigma_{\tau})\cap \ov{\tau})\cap\ov{\tau}\subset\pi_\eps(g(\sigma_{\tau}))\cap\ov{\tau}\subset\pi_\eps(g(\sigma_{\tau})).
\end{align*}
As $\ov{\tau}\subset \pi_{\eps}(g(\sigma_{\tau}))$ and $\pi_{\eps}(g(\sigma_{\tau}))$ is compact, we conclude
$$
\tau=\cl(\ov{\tau})\subset\cl(\pi_{\eps}(g(\sigma_{\tau})))=\pi_{\eps}(g(\sigma_{\tau})),
$$
as required.
\end{proof}


\section{Differentiable approximation of continuous definable maps}

We make use of Proposition \ref{key3} and Proposition \ref{key} to deduce Theorem \ref{main} as a consequence of Paw\l{}ucki's desingularization (see Theorem \ref{pawudesing}).

\begin{proof}[Proof of Theorem \ref{main}]

Fix $\veps>0$ and an integer $p\geq1$. As $X$ is compact and $f$ a continuous definable map, the set $Y:=f(X)\subset\R^m$ is a compact definable set. By Paw\l{}ucki's desingularization (see Theorem \ref{pawudesing}), there exists a definable triangulation $(\Ll, \psi)$ of $Y$ in $\R^m$ such that $\psi:|\Ll|\to Y$ is a definable homeomorphism of class $\Cont^p$. The map $\psi_*:\Cont^0(X,|\Ll|)\to \Cont^0(X,Y), G\mapsto \psi\circ G$ is continuous with respect to the compact-open topology. Thus, there exists $\delta_0>0$ such that if $g\in \Cont^0(X,|\Ll|)$ is a definable map of class $\Cont^p$ such that $\|\psi^{-1}\circ f-g\|<\delta_0$, then the definable map $\psi\circ g: X\to Y$ satisfies
$$
\|f-\psi\circ g\|=\|\psi\circ\psi^{-1}\circ f-\psi\circ g\|<\veps.
$$
In particular, up to substituting $f$ with $\psi^{-1}\circ f$, we may assume: \textit{$Y=|\Ll|$ is the (underlying) polyhedron of a finite simplicial complex $\Ll$ of $\R^m$.} Moreover, by \cite[Thm.15.4]{mu}, there exists an integer $\ell\geq 0$ such that $\diam(\tau)<\tfrac{\veps}{6}$ for each $\tau\in\sd^\ell(\Ll)$. Thus, up to substituting $\Ll$ with $\sd^\ell(\Ll)$, we may also assume assume: \textit{$\diam(\tau)<\tfrac{\veps}{6}$ for each $\tau\in\Ll$.}

By \cite[Cor.8.3.9]{vD}, there exists an open definable neighbourhood $V\subset\R^m$ of $|\Ll|$ and a definable retraction $\rho:\cl(V)\to |\Ll|$. As $|\Ll|$ is compact, up to shrinking $V$ if necessary, we may assume that  $V=\{x\in \R^m : d(x,|\Ll|)<\mu_0\}$ for some $\mu_0>0$. By the Tietze extension theorem for definable maps \cite[Cor.8.3.10]{vD}, the continuous definable map $f:X\to |\Ll|$ can be extended to a continuous definable map $\widehat{f}:\R^n\to \R^m$. We claim: \textit{There exists a compact polyhedron $Q\subset \R^n$ such that $X\subset \Int(Q)$ and $\widehat{f}(Q)\subset V$.} 

As $X$ is compact, there exists $R>0$ such that $X\subset (-R,R)^n$. Let 
$$
U:=\widehat{f}^{-1}(V)\cap (-R,R)^n.
$$
As $X$ is compact, $U$ is open and $X\subset U$, then $\nu:=\dist(X,\partial U)>0$. In particular, the set
$$
U':=\{x\in \R^n : d(x,X)<\nu\}
$$
is an open definable neighbourhood of $X$ in $U$. Consider a triangulation $\Qq_0$ of $[-R,R]^n$ such that $\diam(\sigma)<\nu$ for each $\sigma\in \Qq_0$. Define the compact polyhedron 
$$
Q:=\bigcup\{\sigma\in \Qq_0 : \sigma\subset U'\}.
$$ 
Observe that, as $Q\subset\widehat{f}^{-1}(V)$, then $\widehat{f}(Q)\subset V$. Let $x\in X$ and $B\subset \R^n$ the open ball of centre $x$ and radius $\tfrac{\nu}{2}$. Let $\{\sigma_1,\ldots,\sigma_k\}$ be the set of simplices of $\Qq_0$ such that $B\cap \sigma_i\neq \varnothing$. For each $i=1,\ldots,k$ and $y\in \sigma_i$ we have $d(y,X)\leq d(y,x)<\nu$. In particular, $\sigma_i\subset Q$ for each $i=1,\ldots,k$. As $B\subset [-R,R]^n$, then
$$
B=\bigcup_{\sigma\in\Qq_0}(B\cap \sigma)=\bigcup_{i=1}^k(B\cap \sigma_i)\subset Q,
$$ 
so $x\in\Int(Q)$. We conclude that $X\subset \Int(Q)$, as desired. 

In order to lighten the notation, we replace $\widehat{f}$ by $\widehat{f}|_{Q}$. As $\widehat{f}(Q)\subset V$, the definable map $F:=\rho\circ \widehat{f}:Q\to |\Ll|$ is continuous and well-defined. By \cite[Thm.8.2.9]{vD}, there exists a definable triangulation $(\pol,\varphi)$ of $Q$ compatible with $X$. As the triangulation $(\pol,\varphi)$ is compatible with $X$ and $X$ is compact, the set $\Tt:=\{\sigma\in\pol : \sigma\subset\varphi^{-1}(X)\}$ is a subcomplex of $\pol$. The map 
$$
\rho_*:\Cont^0(|\pol|,V)\to\Cont^0(|\pol|,|\Ll|),\, G\mapsto \rho\circ G
$$ 
is continuous with respect to the compact-open topology, thus, there exists $\delta'_1>0$ such that if $h\in\Cont^0(|\pol|,V)$ is a continuous map such that $\|\widehat{f}\circ \varphi-h\|<\delta'_1$, then
$$
\|F\circ\varphi-\rho\circ h\|=\|\rho\circ \widehat{f}\circ \varphi-\rho\circ h\|<\frac{\veps}{3}.
$$
As $|\pol|$ is compact and $\widehat{f}(\varphi(|\pol|))=\widehat{f}(Q)\subset V$, then $\mu_1:=\dist(\widehat{f}(\varphi(|\pol|),\partial V)>0$. We define $\delta_1:=\min\{\delta_1',\mu_1\}$. As 
$$
\widehat{f}(\varphi(|\Tt|))=\widehat{f}(X)=f(X)=|\Ll|,
$$
by Proposition \ref{key3}, there exists two integers $\kappa,\ell\geq 0$ and a surjective simplicial map $h:|\sd^\kappa(\Tt)|=|\Tt|\to|\sd^{\ell}(\Ll)|$ such that $\|\widehat{f}\circ\varphi|_{|\Tt|}-h\|<\tfrac{\delta_1}{2}$. We claim: \textit{There exists a continuous definable extension $\widehat{h}:|\pol|\to V$ of $h$ such that $\|\widehat{f}\circ\varphi-\widehat{h}\|<\delta_1$.}

If $\widehat{h}:|\pol|\to\R^m$ is a continuous definable map such that $\|\widehat{f}\circ\varphi-\widehat{h}\|<\delta_1$, then 
$$
|\widehat{f}(\varphi(x))-\widehat{h}(x)|_m<\delta_1\leq\mu_1
$$ 
for each $x\in |\pol|$. In particular $\widehat{h}(|\pol|)\subset V$. As $|\pol|$ is compact, the map $\widehat{f}\circ \varphi$ is uniformly continuous. In particular, there exists $\delta^*>0$ such that if $x,y\in|\pol|$ are such that $|x-y|_n<\delta^*$, then 
$$
|\widehat{f}(\varphi(x))-\widehat{f}(\varphi(y))|_m<\frac{\delta_1}{2}.
$$
By \cite[Cor.8.3.9]{vD}, there exists an open definable neighbourhood $W\subset|\pol|$ of $|\Tt|$ and a definable retraction $\xi:\cl(W)\to |\Tt|$. Observe that as $|\pol|$ is compact, also $\cl(W)$ is compact. The set $W':=\{x\in W : |x-\xi(x)|_n<\tfrac{\delta^*}{2}\}$ is an open definable neighbourhood of $|\Tt|$ in $|\pol|$. Thus, up to replacing $W$ with $W'$, we may assume $|x-\xi(x)|_n<\delta^*$ for each $x\in \cl(W)$. Let $\{\theta_1,\theta_2\}$ be a definable continuous partition of unity subordinated to the open covering $\{W, |\pol|\setminus |\Tt|\}$ of $|\pol|$. For each $x\in|\pol|$ define 
$$
\widehat{h}(x):=\theta_1(x)h(\xi(x))+\theta_2(x)\widehat{f}(\varphi(x)).
$$
Observe first that, as $\theta_1\equiv 0$ on $|\pol|\setminus W$, then $\widehat{h}:|\pol|\to\R^m$ is a well-defined continuous definable map. Let $x\in |\pol|$. If $x\in |\pol|\setminus W$, then $\widehat{h}(x)=\widehat{f}(\varphi(x))$, so $|\widehat{f}(\varphi(x))-\widehat{h}(x)|_m=0$. If $x\in W$, then
\begin{align*}
|\widehat{f}(\varphi(x))&-\widehat{h}(x)|_m=|\widehat{f}(\varphi(x))-(\theta_1(x)h(\xi(x))+\theta_2(x)\widehat{f}(\varphi(x)))|_m\\
&=|(1-\theta_2(x))\widehat{f}(\varphi(x))-\theta_1(x)h(\xi(x))|_m=|\theta_1(x)\widehat{f}(\varphi(x))-\theta_1(x)h(\xi(x))|_m\\
&\leq \theta_1(x)|\widehat{f}(\varphi(\xi(x)))-h(\xi(x))|_m+\theta_1(x)|\widehat{f}(\varphi(x))-\widehat{f}(\varphi(\xi(x)))|_m\\
&<\max_{x\in|\pol|}\{\theta_1(x)\}\|\widehat{f}\circ\varphi|_{|\Tt|}-h\|+\frac{\delta_1}{2}\leq\|\widehat{f}\circ\varphi|_{|\Tt|}-h\|+\frac{\delta_1}{2}\\
& <\frac{\delta_1}{2}+\frac{\delta_1}{2}=\delta_1.
\end{align*}
We conclude that $\|\widehat{f}\circ\varphi-\widehat{h}\|<\delta_1$, as desired.

Define $H:=\rho\circ \widehat{h}:|\pol|\to |\Ll|$ and observe that $H$ is surjective (because $h$ is surjective) and 
\begin{equation}\label{main1}
\|F\circ \varphi-H\|=\|\rho\circ\widehat{f}\circ\varphi-\rho\circ\widehat{h}\|<\frac{\veps}{3},
\end{equation}
because $\|\widehat{f}\circ\varphi-\widehat{h}\|<\delta_1\leq \delta_1'$. By Proposition \ref{key}, there exists $\eps>0$ with the following property: if $g:|\Tt|\to|\Ll|$ is a continuous map such that $\|h-g\|<\eps$, then $\pi_{\eps}\circ g:|\Tt|\to|\Ll|$ is surjective and 
$$
\|h-\pi_{\eps}\circ g\|<2\max_{\tau\in\sd^\ell(\Ll)}\{\diam(\tau)\}\leq2\max_{\tau\in\Ll}\{\diam(\tau)\} <\frac{\veps}{3},
$$ 
where $\pi_{\eps}$ is the $\eps$-squeezing map of $\sd^\ell(\Ll)$ (see Definition \ref{defsqueezing}). As $Q$ is compact, the maps $F$ and $H\circ\varphi^{-1}$ are uniformly continuous. Thus, there exists $\delta_2>0$ such that if $x,y\in Q$ are such that $|x-y|_n<\delta_2$, then 
\begin{equation}\label{1}
|F(x)-F(y)|_m<\frac{\veps}{3} \text{\, and \, } |H(\varphi^{-1}(x))-H(\varphi^{-1}(y)) |_m<\eps.
\end{equation}

Let $(\Qq,\id_Q)$ be a triangulation of $Q$ such that $\diam(\sigma)<\delta_2$ for each $\sigma\in\Qq$. By Paw\l{}ucki's desingularization (see Theorem \ref{pawudesing}), there exists a $\Cont^p$-triangulation $(\Qq^*,\Phi)$ of $|\Qq|=Q$ such that:
\begin{itemize}
\item $\Qq^*$ is a refinement of $\Qq$,
\item $\Phi(\sigma)=\sigma$, for each simplex $\sigma\in\Qq$,
\item $\pi_{\eps}\circ H\circ\varphi^{-1}\circ\Phi$ is of class $\Cont^p$.
\end{itemize}
The (of course not commutative) diagram in Figure \ref{notcommutative} summarizes the situation. 
\begin{figure}[ht!]
$$
\xymatrix{
 	&				&		&	  &   Y &	& 	 \\
& &  X\ar@{^{(}->}[d]\ar@{^{(}->}[dll]\ar[rr]^{\psi^{-1}f} \ar[urr]^{f}& &  |\Ll|\ar[u]^{\psi}\ar[u]^{\psi}\ar@{=}[rr]&& |\Ll|\ar[ull]_{\psi}\\
|\Qq^*|\ar[rr]^{\Phi} & & |\Qq|\ar[rr]^{\widehat{f}} \ar[urr]^{F}& & V\ar[u]_{\rho}\ar[rr]^{\rho} & & |\Ll|\ar[u]^{\pi_{\eps}} \\
& & |\pol|\ar[u]^{\varphi}\ar[rrrr]^{\widehat{h}}\ar[urrrr]^{H}& & & & V	\ar[u]^{\rho}\\
& & |\Tt|\ar@{^{(}->}	[u]\ar[uull]_{\varphi|_{|\Tt|}}\ar[rrrr]^{h} & && & |\Ll|\ar@{^{(}->}	[u] 
}
$$
\caption{\small{A diagram that summarizes the current situation.}}
\label{notcommutative}
\end{figure}
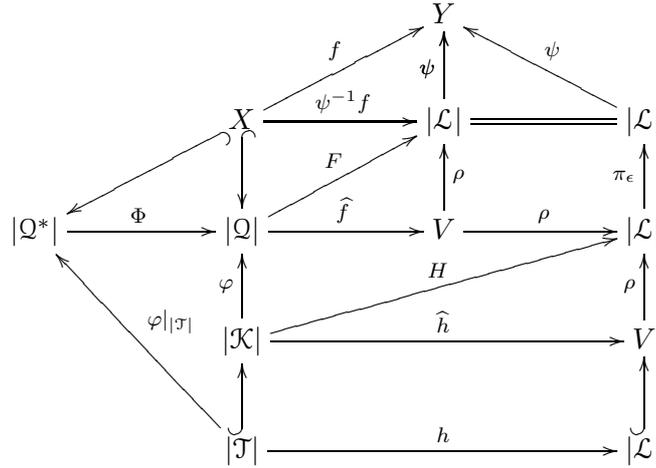

Define $g:=\pi_{\eps}\circ H\circ \varphi^{-1}\circ\Phi|_{X}:X\to |\Ll|$, which is a definable map of class $\Cont^p$. Observe first that, as $\Phi(\sigma)=\sigma$ and $\diam(\sigma)<\delta_2$ for each simplex $\sigma\in\Qq$, then $|x-\Phi(x)|_n<\delta_2$ for each $x\in|\Qq^*|=Q$. By \eqref{1}, we deduce
\begin{align*}
\|h-H\circ\varphi^{-1}\circ\Phi\circ\varphi|_{|\Tt|}\|&=\|H\circ\varphi^{-1}\circ\varphi|_{|\Tt|}-H\circ\varphi^{-1}\circ\Phi\circ\varphi|_{|\Tt|}\|\\
&=\|H\circ\varphi^{-1}|_X-H\circ\varphi^{-1}\circ\Phi|_X\|\\
&\leq\max_{x\in Q}|H(\varphi^{-1}(x))-H(\varphi^{-1}(\Phi(x)))|_m <\eps.
\end{align*}
Thus, by Proposition \ref{key},
$$
g(X)=\pi_{\eps}(H(\varphi^{-1}(\Phi(X))))=\pi_{\eps}(H(\varphi^{-1}(\Phi(\varphi(|\Tt|)))))=|\Ll|.
$$
In order to conclude, it only remains to show that the map $g$ satisfies $\|f-g\|<\veps$. First observe that, as $\pi_\eps$ is the $\eps$-squeezing map of $\sd^{\ell}(\Ll)$, then $\pi_{\eps}(\tau)=\tau$ for each simplex $\tau\in\sd^\ell(\Ll)$ (see Remark \ref{partialpi}), so
\begin{equation}\label{main2}
|x-\pi_\eps(x)|_m\leq\max_{\tau\in\sd^\ell(\Ll)}\{\diam(\tau)\}\leq\max_{\tau\in\Ll}\{\diam(\tau)\}<\frac{\veps}{3}
\end{equation}
for each $x\in|\Ll|$. As $\Phi:Q\to Q$ is a homeomorphism such that $|x-\Phi(x)|_n<\delta_2$ for each $x\in Q$, then also $|\Phi^{-1}(y)-y|_n<\delta_2$ for each $y\in Q$. In particular, by \eqref{1}, we have $\|F\circ\Phi^{-1}-F\|<\tfrac{\veps}{3}$. By \eqref{main1} and \eqref{main2}, we deduce, 
\begin{align*}
\|f-g\|&=\|F|_X-\pi_{\eps}\circ H\circ \varphi^{-1}\circ\Phi|_{X}\|\leq \|F-\pi_{\eps}\circ H\circ \varphi^{-1}\circ\Phi\|\\
&\leq \|F-H\circ \varphi^{-1}\circ\Phi\|+\|H\circ \varphi^{-1}\circ\Phi-\pi_{\eps}\circ H\circ \varphi^{-1}\circ\Phi\|\\
&<\|F\circ\Phi^{-1}\circ\Phi-H\circ\varphi^{-1}\circ\Phi\|+\frac{\veps}{3}=\|F\circ\Phi^{-1}-H\circ\varphi^{-1}\|+\frac{\veps}{3}\\
&\leq \|F\circ\Phi^{-1}-F\|+\|F-H\circ\varphi^{-1}\|+\frac{\veps}{3}\\
&< \frac{\veps}{3}+\|F\circ\varphi\circ\varphi^{-1}-H\circ\varphi^{-1}\|+\frac{\veps}{3}\\
&=\|F\circ\varphi-H\|+\frac{2\veps}{3}<\frac{\veps}{3}+\frac{2\veps}{3}=\veps,
\end{align*}
as required.
\end{proof}

\subsection*{Acknowledgments} The author would like to thank Riccardo Ghiloni for suggesting to investigate this problem and for valuable discussions during the preparation of this work.

\bibliographystyle{amsalpha}

\end{document}